\documentclass[12pt,a4paper]{amsart}
\usepackage[backref,colorlinks,linkcolor=red,anchorcolor=green,citecolor=blue]{hyperref}
\usepackage{amsfonts}
\usepackage{amsthm}
\usepackage{amsmath}
\usepackage{amscd}
\usepackage{t1enc}
\usepackage{amssymb}
\usepackage{mathrsfs}

\usepackage[shortlabels]{enumitem}
\usepackage{pgfplots}
\usepackage{xcolor}
\usepackage{color,soul}

\usepackage{color,soul}
\usepackage[mathscr]{eucal}
\usepackage{indentfirst}
\usepackage{graphicx}
\usepackage{graphics}
\usepackage{pict2e}
\usepackage{hyperref}
\usepackage{epic}
\usepackage{tikz-cd}
\numberwithin{equation}{section}
\usepackage[a4paper,top=3cm,bottom=3cm,left=1cm,right=1cm,marginparwidth=1.65cm]{geometry}
\usepackage{verbatim}

\begin{document}
\newtheorem{thm}{Theorem}[section]
\newtheorem{remark}{Remark}
\newtheorem{cor}[thm]{Corollary}
\newtheorem{lem}[thm]{Lemma}
\newtheorem{prop}[thm]{Proposition}

\newcommand{\norm}[1]{\left\Vert#1\right\Vert}
\newcommand{\abs}[1]{\left\vert#1\right\vert}
\newcommand{\set}[1]{\left\{#1\right\}}
\newcommand{\Real}{\mathbb R}
\newcommand{\eps}{\varepsilon}
\newcommand{\To}{\longrightarrow}
\newcommand{\BX}{\mathbf{B}(X)}
\newcommand{\A}{\mathcal{A}}
\newcommand{\ts}{\textstyle}
\newcommand{\tg}{\mbox{\rm tg}}
\newcommand{\ctg}{\mbox{\rm ctg}}
\newcommand{\atg}{\tg^{-1}}
\newcommand{\actg}{\ctg^{-1}}
\newcommand{\asin}{\sin^{-1}}
\newcommand{\acos}{\cos^{-1}}
\newcommand{\dps}{\displaystyle}
\newcommand{\fnz}{\footnotesize}
\newcommand{\D}{\displaystyle}
\newcommand{\df}[2]{\dfrac{\D#1}{\D#2}}
\newcommand{\scp}{\scriptstyle}

\baselineskip 18pt

\title[]{Viscous shock waves of Burgers equation with fast diffusion and singularity}

\thanks{ Emails: shufangxu@shu.edu.cn (S. Xu),  ming.mei@mcgill.ca (M. Mei), mmei@champlaincollege.qc.ca (M. Mei),  jean-christophe.nave@mcgill.ca (J.-C. Nave),  mathwcsheng@shu.edu.cn (W. Sheng)}
\author[Xu, Mei, Nave and Sheng]{Shufang Xu$^{1,3}$, Ming Mei$^{2,3}$, Jean-Christophe Nave$^{3}$ and Wancheng Sheng$^{1}$}

\dedicatory{$^1$Department of Mathematics, Shanghai University,
            Shanghai, 200444, P. R. China\\
            $^2$Department of Mathematics, Champlain College Saint-Lambert, Saint-Lambert, Quebec, J4P3P2, Canada\\
            $^3$Department of Mathematics and Statistics, McGill University, Montreal, Quebec, H3A2K6, Canada
            }

\noindent

\begin{abstract}
In this paper, we study the asymptotic stability of viscous shock waves for Burgers' equation with fast diffusion $u_t+f(u)_x=\mu (u^m)_{xx}$ on $\mathbb{R} \times (0, +\infty)$ when $0<m<1$. For the proposed constant states $u_->u_+=0$, the equation with fast diffusion $(u^m)_{xx}=m\left(\frac{u_x}{u^{1-m}}\right)_x$ processes a strong singularity at $u_+=0$, which causes the stability study to be challenging. We observe that, there exist two different types of viscous shocks, one is the non-degenerate shock satisfying Lax's entropy condition with fast algebraic decay to the singular state $u_+=0$, which causes  much strong singularity  to the system in the form of $m\left(\frac{u_x}{u^{1-m}}\right)_x$, and the other is the degenerate viscous shock with slow algebraic decay to $u_+=0$, which makes  less strong singularity  to the system. In order to overcome the singularity at $u_+=0$, we technically use the weighted energy method and develop a  new strategy where the weights related to the shock waves are carefully selected, while the chosen weights for the non-degenerate case are stronger than the degenerate case. Numerical simulations are also carried out in different cases to illustrate and validate our theoretical results. In particular, we numerically approximate the solution for different value of $0<m<1$, and find that the shapes of shock waves become steeper when the singularity $\left(\frac{u_x}{u^{1-m}}\right)_x$ is stronger as $m\rightarrow 0$, which indicates that the effect of singular fast diffusion on the solution is essential.

\vskip 8pt

\noindent%

\vskip 8pt
\end{abstract}

\maketitle

\section{Introduction}

This paper is devoted to studying the Cauchy problem for the one-dimensional Burgers' equation with  fast diffusion
\begin{equation}\label{system}
  u_t+f(u)_x=\mu (u^m)_{xx},\;\;x\in \rm{R},\,t>0,
\end{equation}
subjected to the initial data
\begin{equation}\label{ini}
u(0,x)=u_0(x)\rightarrow u_\pm,~\mbox{as}~x\rightarrow\pm\infty,
\end{equation}
where $m\in(0,1)$ is the index of fast diffusion, $\mu>0$ is the diffusion coefficient, and  $u_\pm$ are the constant states satisfying
\begin{equation}
  u_->u_+=0.
\end{equation}
It is clear that equation \eqref{system} with fast diffusion $\mu (u^m)_{xx}=\mu m \left(\df{u_x}{u^{1-m}}\right)_x$ possesses a strong singularity when $u=0$.

Introduced by Bateman \cite{bat} for fluid dynamics, and pioneered by J. M. Burgers \cite{bur1,bur2} for turbulence of fluid dynamics, the so-called Burgers' equation is one of the most fundamental partial differential equations in fluid/gas dynamic. From the point view of applications, the effects of fluid viscosity are related to the velocity $u(t,x)$, for instance in the porous media equations. We give now some details for equation \eqref{system}.

When $m=1$, eq. \eqref{system} possesses \emph{regular diffusion}:
\[
u_t +f(u)_x =\mu u_{xx}.
\]
When $m>1$, eq. \eqref{system} has \emph{slow diffusion}, which is usually called the degenerate diffusion:
\[
u_t +f(u)_x =\mu m (u^{m-1} u_{x})_x.
\]
The solution for this type of equation usually forms sharp corners and losses its regularity \cite{GK}. When $m<1$, eq. \eqref{system} reduces to the \emph{fast diffusion} equation:
\[
u_t +f(u)_x =\mu m (\frac{u_{x}}{u^{1-m}})_x.
\]
Here, the case of $m=0$ corresponds to the critical fast diffusion:
\begin{equation}\label{critical-fast-diffusion}
u_t +f(u)_x =\mu  (\frac{u_{x}}{u})_x = \mu (\ln u)_{xx}, \ \ \ \mbox{ for } u\ge 0,
\end{equation}
and finally, the case for which $(\frac{u_{x}}{u^{1-m}})_x$ with $m<0$ is called the \emph{super-fast diffusion}.

The focus of the present paper is to investigate the nonlinear stability of viscous shock waves $U(x-st)$ with two constant states $u_->u_+=0$, where $s$ is the speed of shock waves, and $u_+=0$ is the singular state, because the fast diffusion $(\frac{u_x}{u^{1-m}})_x$ with $0<m<1$ will be $\infty$ at $u_+=0$.

{\bf Background of study}. Regarding the stability of viscous shock waves, it has been one of hot research spots in fluid dynamical PDEs, and has been extensively studied.

For Burgers' equation with regular diffusion of  $m=1$, Ili'in and Oleinik \cite{Ili'in}  first studied early in 1960 the stability of shock profiles by the maximum principle, then in 1976 Sattinger \cite{spectural} obtained stability via spectral analysis of the differential operator. Later in 1985, Kawashima and Matsumura \cite{same line} and Nishihara \cite{Nishihara} proved the stability using standard $L^2$-energy method when the flux function $f(u)$ is convex. When the flux function is nonconvex, Kawashima and Matsumura \cite{nonconvex}, Mei \cite{M. Mei}, Matsumura and Nishihara \cite{conventional energy method} further showed stability of shock waves using the $L^2$-weighted energy method. Later,  Freist\"uhler and Serre \cite{Serre} presented the $L^1$-stability by the semigroup method, and Howard \cite{Howard-3, Howard-2,Howard-1} showed  stability via the Evan function method. Very recently, Kang and Vasseur \cite{KV} proved the same stability by the relative entropy method. For other significant contributions, we refer the reader to  Weinberger \cite{nonconvex3}, Jones \textit{et al.} \cite{Jones}, Engler \cite{Engler}, Fries \cite{Fries}, as well as for the system case by Matsumura-Nishihara \cite{Matsumura-Nishihara-1}, Goodman \cite{Goodman},  Liu \cite{Liu}, Liu-Zeng \cite{Liu-Zeng-1, Liu-Zeng-2}, Macia-Zumbrun \cite{MZ}, Szepessy-Xin \cite{Szepessy}, Kang-Wang-Vassour \cite{KVW}. Additionally, the case with boundary effect was investigated by Liu and collaborators \cite{IBVP2, IBVP3, IBVP1} and Matsumura-Mei \cite{Matsumura-Mei}.

Regarding the slow (degenerate) diffusion case with $m>1$ and the fast diffusion case with $m<1$ for Burgers' equation, the stability of shock waves is quite challenging and almost never related. In fact, in the slow diffusion case with $m>1$, the viscous shock waves become sharp, non-differentiable, and lose their regularity, which pauses an essential difficulty for the stability proof. As is well known, there are only four works related to the stability of sharp traveling waves, two are by Bir$\acute{\mbox{o}}$ \cite{Biro} and  Kamin-Rosenau \cite{K-R-1} for Fisher-KPP equations, and the other two  by Kamin-Rosenau \cite{K-R-2} and Xu-Ji-Mei-Yin \cite{XJMY} for Burgers-Fisher-KPP equations.  All of them are based on the monotonic technique by constructing a pair of upper and lower solutions in very special forms. The effect of reaction terms for these Fisher-KPP equations plays a key role for the proof of stability. However, for Burgers' equation \eqref{system}, there is no such reaction term, and the special monotonic technique cannot be applied to Burgers' equations with degenerate diffusion. The stability of sharp viscous shock waves is still an open problem, as far as the authors know.

For the fast diffusion case that is with $m<1$, the main difficulty is the strong singularity coming from the fast diffusion of $(\frac{u_x}{u^{1-m}})_x$ at the singular state $u_+=0$.
Recently, for the critical fast diffusion case \eqref{critical-fast-diffusion} with $m=0$, joint with X. Li and J. Li, the second and the third authors \cite{Li-Li-Mei-Nave} first proved the stability of shock waves with the singularity at $u_+=0$ by the weighted energy method, where in order to overcome the singularity of the diffusion at the constant state $u_+=0$, the weight function related to the shock wave was carefully constructed.

{\bf Difficulties and strategies}. The present article is a follow-up of \cite{Li-Li-Mei-Nave} for the fast diffusion case with $0<m<1$.  In order to tackle this problem, we first need to analyze the property of viscous shock waves $U(x-st)$ ($s$ is the wave speed) with singular point $u_+=0$ based on the value of $0<m<1$, then we use the weighted-energy method to show the stability of these shock waves. We observe that, there exist two different types of viscous shock waves based on different setting of flux function $f(u)$ and the states $u_\pm$. One is the regular (non-degenerate) viscous shock satisfying Lax's entropy condition, with fast algebraic decay to the singular state $u_+=0$. This causes a much strong singularity to the system in the form of $m(\frac{u_x}{u^{1-m}})_x$. The other is the degenerate shock with slow algebraic decay to $u_+=0$, which makes a less strong singularity to the equation.  In order to overcome the singularity at $u_+=0$, we  adopt the technical weighted-energy method to treat the asymptotic stability. Since the singularities for the solution itself and its derivatives are totally different, the weight functions for treating these singularities are also different, and should be  carefully chosen based on the different index $m$ for the fast diffusion for both non-degenerate and degenerate cases of shock waves. In particular, with the new observation mentioned before,  the weights for the non-degenerate case need to be  much stronger than the degenerate case. On the other hand, different from the study \cite{Li-Li-Mei-Nave},  we realize that the effect of the flux function plays a crucial role in the stability, see Theorem \ref{Mthm} in Section 2 and the typical examples in  Section 6, and also that the initial data are restricted to the index $m$. To treat the singularities at the far fields when we take the weighted energy estimates, we artfully apply the cut-off technique. On the other hand, we carry out some numerical simulations for these examples at the end of the paper. Our numerical experiments confirm and illustrate our theoretical results. In particular, we observe from the numerical experiment that the shape of solutions to \eqref{system} and \eqref{ini} steepen when $m$ gets close to $0^+$. This indicates that the effect of the singular fast diffusion to the solution is essential.

This paper is organized as follows. In Section 2, we present the main results of this paper, including the existence and the nonlinear stability of the shock profile. In Section 3, we present some preliminary calculations and derive the perturbation equation. Section 4 is devoted to the proof of existence. Section 5 contains the proof for the nonlinear stability of viscous shock waves in the non-degenerate case. Section 6 contains the proof for the nonlinear stability of viscous shock waves in the degenerate case. Finally, in Section 7, we provide four examples and carry out some numerical simulations to illustrate our results.

We finish by introducing some notation.

\noindent
{\bf Notation.}~~In this paper, we denote  the generic positive constants which are independent of the time $t$  by $C$ unless otherwise stated. We denote $f(x)\sim g(x)$ as $x\rightarrow x_0$ when $C^{-1}g\le f \le Cg$ in a neighborhood of $x_0$. For function spaces,
  $L^p$ denotes the space of $p$-th integrable functions on $R$ with the usual norm $\|\cdot\|_{L^p}$. For simplicity, $\|\cdot\|_{L^2}=\|\cdot\|$. $H^l(l\geq 0)$ denotes the usual $l^{th}$ order Sobolev space with the standard norm $\|\cdot\|_l$. For the weight function $w>0$, $L_w^2$ denotes the space of measurable functions $f$ satisfying $\sqrt{w}f\in L^2$ with the norm $\|\cdot\|_w$, meanwhile, $H_w^l$ denotes weighted Sobolev space of $f$ satisfying $\sqrt{w}\partial_x^j f\in L^2$ for $0\leq j\leq l$ with the norm $\|\cdot\|_{l,\omega}$.
Denote
\begin{equation}
\langle x\rangle =\sqrt{1+x^2}, \mbox{ and }
  \langle x\rangle_+=\left\{
  \begin{array}{ll}
    \sqrt{1+x^2},~&\mbox{for}~x\geq0,\\
    1,&\mbox{for}~x<0.
  \end{array}
  \right.
\end{equation}
 The weighted space $L^2_w$ for such weight function $w=\langle x\rangle_+^\alpha$ is denoted as $L^2_{\langle x\rangle_+^\alpha}$, and the corresponding norm is $\|\cdot\|_{\langle x\rangle_+^\alpha}$. The weighted Sobolev space $H_{\langle x\rangle_+^\alpha}^l$ are defined similarly.

\section{Main results}

A viscous shock wave to  the system \eqref{system} is the special solution in the form of
\begin{equation}\label{sp}
u(x,t)=U(x-st)=: U(\xi),\;\;U(\xi)\rightarrow u_\pm\; \mbox{ as } \;\xi\rightarrow \pm\infty,
\end{equation}
where $0=u_+<u_-$ and $s$ is the speed of shock profiles.  Namely, it satisfies
   \begin{equation}\label{xi}
    \left\{
    \begin{array}{ll}
    -sU_\xi+f(U)_\xi=\mu(U^m)_{\xi\xi},\\
    U(\pm\infty)=u_\pm,
    \end{array}\right.
  \end{equation}
  where $\partial_\xi=\frac{\rm d}{{\rm d}\xi}$.
 Integrating the first equation of \eqref{xi} in $\xi$ over
$(\pm\infty,\xi)$, we have
\begin{equation}\label{Hu}
\dfrac{\mu m}{U^{1-m}}U_\xi
=-s(U-u_\pm)+f(U)-f(u_\pm)=: g(U).
\end{equation}
Denote
\begin{equation}
 U_\xi
=\dfrac{U^{1-m}}{\mu m}[-s(U-u_\pm)+f(U)-f(u_\pm)]=:h(U),
\end{equation}
 provided that, by passing $\xi\to\pm\infty$ and formally confirming $U_\xi(\pm\infty)=0$, then the state constants $u_\pm$ and $s$ satisfy the Rankine-Hugoniot condition
 \begin{equation}\label{RH}
s=\df{f(u_+)-f(u_-)}{u_+-u_-}.
\end{equation}
When we look for the monotonic viscous shock $u_+=0<U(x-st)<u_-$, namely, $U_\xi<0$, we then need the following generalized  entropy condition
  \begin{equation}\label{Sc}
    g(U)<0,~\mbox{for}~U\in(u_+,u_-),
  \end{equation}
 equivalently,
 \[
 -s(U-u_\pm)+f(U)-f(u_\pm)<0, ~\mbox{for}~U\in(u_+,u_-).
 \]
This implies that
  \begin{equation}
     f'(u_+)\leq s \leq f'(u_-),
  \end{equation}
  where $~'=\frac{\rm d}{{\rm d}u}$.
When $s=f'(u_+)$ or $s=f'(u_-)$, we call the viscous shock wave $U(x-st)$ to be degenerate. The viscous shock wave $U(x-st)$ is said to be regular or non-degenerate, if Lax's entropy condition holds
\begin{equation}\label{lax}
 f'(u_+)<s<f'(u_-).
\end{equation}
In the degenerate case, we mainly consider the significant case  of $s=f'(u_+)$, because $u_+=0$ is the singular state for the equation. We assume that
  \begin{equation}\label{1-11-1}
   f''(u_+)=\cdots =f^{(k_+)}(u_+)=0~\mbox{and}~f^{(k_++1)}(u_+)\neq 0,~\mbox{ for some integer }~k_+\geq 1.
  \end{equation}

Now we state the existence of shock profiles and their stability.

\begin{thm}[Existence of shock profiles]\label{pro}
Assume  $u_->u_+=0$.
\begin{itemize}[leftmargin=*]
 \item [$\bullet$] Necessary condition: if \eqref{system} admits a monotonically decreasing shock profile $U(x-st)$ connecting  $u_-$ and $u_+$, then $u_\pm$ and $s$ must satisfy the Rankine-Hugoniot condition \eqref{RH} and the generalized shock condition \eqref{Sc}.
  \item [$\bullet$] Sufficient condition: suppose that \eqref{RH} and \eqref{Sc} hold, then there exists a monotonically decreasing solution  $U(\xi)$ of \eqref{system} with $U(\pm\infty)=u_\pm$, which  is  unique up to a shift in $\xi$.
  \item [$\bullet$] Decay properties: for the non-degenerate case  $f'(u_+)<s<f'(u_-)$, it holds that
 \begin{equation}\label{210}
 \begin{array}{lll}
    |U_\xi|^{\frac{1}{2-m}} \sim |U(\xi)-u_+|\sim |\xi|^{-\frac{1}{1-m}},~\mbox{as}~\xi\rightarrow +\infty,\\[4pt]
    |U_\xi|\sim |U(\xi)-u_-|\sim \exp ({-\lambda|\xi|}),~\mbox{as}~\xi\rightarrow -\infty,
  \end{array}
 \end{equation}
where $\lambda:=\dfrac{u_-^{1-m}}{\mu m}\left( f'(u_-)-s \right)>0 $.

For the degenerate cases: $f'(u_+)=s<f'(u_-)$ satisfying \eqref{1-11-1},   it holds that
\begin{equation}\label{211}
  \begin{array}{ll}
    |U_\xi|^{\frac{1}{k_++2-m}} \sim |U(\xi)-u_+|\sim |\xi|^{-\frac{1}{k_++1-m}},~\mbox{for}~s=f'(u_+),~\mbox{as}~\xi\rightarrow +\infty,\\[4pt]
    |U_\xi|\sim |U(\xi)-u_-|\sim \exp ({-\lambda|\xi|}),~\mbox{as}~\xi\rightarrow -\infty,
  \end{array}
 \end{equation}
where $k_+\geq 1$.

\end{itemize}
\end{thm}

 For a given shock wave $U(x-st)$, as showed in \cite{same line,nonconvex}, we can determine our targeted shock profile $U(x-st+x_0)$ related to the given initial data by
 \begin{equation}
   x_0:=\df{1}{u_+-u_-}\int^{+\infty}_{-\infty} [u_0(x)-U(x)]{\rm d} x,
 \end{equation}
 such that the so-called ``zero-mass'' condition holds
 \begin{equation}\label{shift}
   \int^{+\infty}_{-\infty}[u_0(x)-U(x+x_0) ]{\rm d}x=0.
 \end{equation}
 Without loss of generality, let us take $x_0=0$. Define
 \begin{equation}
  \phi_0(\xi)=\int_{-\infty}^\xi[u_0(x)-U(x)]{\rm d}x.
\end{equation}

With the existence of shock profiles,  the stability of shock waves is stated as follows.

\begin{thm}[Stability of non-degenerate shock profiles with $f'(u_+)<s<f'(u_-)$, much strong singularity case]\label{Mthm}
Let
\begin{equation}
   K(u):=\df{g(u)}{u^{2m}}
 \end{equation}
be convex on $[u_+,u_-]$, namely,
\begin{equation}\label{M}
  K''(u)\geq 0,~\mbox{for}~u\in[u_+,u_-].
\end{equation}
Let $U(\xi)$ be a viscous shock wave satisfying \eqref{RH}, the Lax's entropy condition \eqref{lax} and \eqref{shift}. $f(u)$ is a smooth function. Let the initial perturbation be $\phi_0\in L^2_{\langle \xi\rangle_+^{\alpha_2}}$ and $\phi_{0\xi}\in H_{\langle \xi\rangle_+^{\alpha_1}}^1$, then
there exists a positive constant $\delta_0$ such that if
\begin{equation}\label{N0}
\left\| {\phi_{0\xi\xi}}\right\|_{\langle \xi\rangle_+^{\alpha_1}}^2+\left\| {\phi_{0\xi}}\right\|_{\langle \xi\rangle_+^{\alpha_1}}^2+\left\| {\phi_0}\right\|_{\langle \xi\rangle_+^{\alpha_2}}^2\leq \delta_0,
\end{equation}
then for the Cauchy problem \eqref{system} and
\eqref{ini} there exists a unique global solution $u(x,t)$ satisfying
\begin{equation}
\begin{array}{ll}
  &u-U\in C^0([0,+\infty);H^1_{\langle \xi\rangle_+^{\alpha_1}})\cap L^2([0,+\infty);L^2_{\langle \xi\rangle_+^{\alpha_4}}),\\[6pt]
  &(u-U)_x\in  L^2([0,+\infty);H^1_{\langle \xi\rangle_+^{\alpha_3}})
\end{array}
\end{equation}
and the asymptotic stability
\begin{equation}
  \sup_{x\in R}|u(t,x)-U(x-st)|\rightarrow0, \;\;as~t\rightarrow +\infty,
\end{equation}
where $\alpha_1=\frac{2}{1-m}$ and $\alpha_2=\frac{2m}{1-m}$, $\alpha_3=\frac{3-m}{1-m}$ and $\alpha_4=\frac{1+m}{1-m}$.
\end{thm}

\begin{thm}[Stability of degenerate shock profiles with $f'(u_+)=s<f'(u_-)$, less strong singularity case]\label{ext}
 Let $U(\xi)$ be a viscous shock wave satisfying \eqref{RH}, the degenerate entropy condition $s=f'(u_+)$ and \eqref{shift}, and $f(u)$ be  smooth. When $\phi_0\in L^2_{\langle \xi\rangle_+^{\beta_2}}$ and $\phi_{0\xi}\in H_{\langle \xi\rangle_+^{\beta_1}}^1$, then there exist a positive constant $\delta_0$ such that if
\begin{equation}\label{N0}
\left\| {\phi_{0\xi\xi}}\right\|_{\langle \xi\rangle_+^{\beta_1}}^2+\left\| {\phi_{0\xi}}\right\|_{\langle \xi\rangle_+^{\beta_1}}^2+\left\| {\phi_0}\right\|_{\langle \xi\rangle_+^{\beta_2}}^2\leq \delta_0,
\end{equation}
then for the Cauchy problem \eqref{system} and
\eqref{ini} there exists a unique global solution $u(x,t)$ satisfying
\begin{equation}
\begin{array}{ll}
  &u-U\in C^0([0,+\infty);H^1_{\langle \xi\rangle_+^{\beta_1}})\cap L^2([0,+\infty);L^2_{\langle \xi\rangle_+}),\\[6pt]
  &(u-U)_x\in  L^2([0,+\infty);H^1_{\langle \xi\rangle_+^{\beta_3}})
\end{array}
\end{equation}
and the asymptotic stability
\begin{equation}
  \sup_{x\in R}|u(t,x)-U(x-st)|\rightarrow0, \;\;as~t\rightarrow +\infty,
\end{equation}
where $\beta_1=\frac{2}{k_++1-m}$, $\beta_2=\frac{k_+}{k_++1-m}$ and $\beta_3=\frac{3-m}{k_++1-m}$.
\end{thm}

\begin{remark}

\begin{enumerate}
\item From  Theorem \ref{pro} for the existence of viscous shocks, the asymptotic behaviour of the viscous shock waves at the far fields is related to the index $m$ of singular fast diffusion.   From  Theorem \ref{Mthm} and Theorem \ref{ext} for the stability of viscous shocks, the initial data is expected to have the same decay rates as the given viscous shock wave at the far fields, namely, the corresponding initial data is also related to the index $m$ of singular fast diffusion. As numerically observed in Section 7, we recognize an interesting phenomenon  that the shape of solutions to \eqref{system} and \eqref{ini} become much steeper when $m$ approaches $0^+$. This indicates that the effect of the singular fast diffusion to the solution is essential.

\item
 The singularity of $\frac{1}{u^{1-m}}$ in the case of $f'(u_+)<s<f'(u_-)$ is much stronger than that in the case of $f'(u_+)=s<f'(u_-)$, and the selected weight functions for the case of $f'(u_+)<s<f'(u_-)$ are also much stronger than the case of $f'(u_+)=s<f'(u_-)$. On the other hand, in the non-degenerate case of $f'(u_+)<s<f'(u_-)$, we need the other restriction \eqref{M} for the stability, while in the degenerate case with less strong singularity, such a restriction is released. This also reflects how the strong singularity affects the stability of shock waves for the fast diffusion equations.

 \item Regarding the sufficient condition \eqref{M} for the stability of non-degenerate shock waves with much strong singularity at $u_+=0$, it involves the index $m$ for the singular fast diffusion, which indicates that the flux
function $f(u)$ is restricted to the singular fast diffusion according to different values of $m$. In fact, the flux function plays a crucial role in the stability of the viscous shock waves. As noted in the regular diffusion case \cite{KV}, to guarantee the wave stability, it is necessary to request the flux function $f(u)$  to satisfy a certain condition (Theorem 1 in \cite{KV}). On the other hand,
as showed later in the last section of numerical simulations, we give two examples for $f(u)$ such that not only the Rankine-Hugoniot \eqref{RH} and the Lax's entropy condition \eqref{lax} are satisfied, but also the crucial condition \eqref{M} is satisfied. For example,
\[
  \begin{array}{ll}
    f(u)=u^2,
  \end{array}
\]
 or
\[
  f(u)=2u^{3+2m}-u^{1+2m},~\mbox{with}~0<m\leq\frac{1}{2},
 \]
the condition \eqref{M} is satisfied for all of these examples.

\end{enumerate}

\end{remark}

\section{Preliminaries and reformulation of the problem}

For convenience in the analysis, transforming the original variables $(x,t)$ to a moving coordinate frame $(\xi=x-st,t)$, we obtain
\begin{equation}
  u_t-su_\xi+f(u)_\xi=\mu (u^m)_{\xi\xi}.
\end{equation}
Set
\begin{equation}
  \phi(\xi,t)=\int_{-\infty}^\xi[u(x,t)-U(x)]{\rm d}x.
\end{equation}
Then with \eqref{xi}, we obtain an equation for $\phi$:
\begin{equation}\label{Main}
\left\{\begin{array}{ll}
  \phi_t+g'(U)\phi_\xi-\left(\mu m\dfrac{\phi_\xi}{U^{1-m}}\right)_\xi=F+G_\xi,\\
  \phi|_{t=0}=\phi_0(\xi),
\end{array}
 \right.
\end{equation}
where
\begin{equation}\label{F}
F:=-\{f(U+\phi_\xi)-f(U)-f'(U)\phi_\xi\},
\end{equation}
and
\begin{equation}\label{G1}
G:=\mu\{(U+\phi_\xi)^m-U^m- m\dfrac{\phi_\xi}{U^{1-m}}\}.
\end{equation}

Define the solution spaces $X(0,T)$ and $Y(0,T)$ for any $0\leq T< +\infty$:

\begin{itemize}[leftmargin=*]
 \item [$\bullet$] for the non-degenerate case of $f'(u_+)<s< f'(u_-)$:
\begin{equation}
\begin{array}{lll}
  X(0,T):=
    &\left\{\phi(\xi,t)\left|\phi \in C^0([0,T];L^2_{\langle \xi\rangle_+^{\alpha_2}}),~\phi_\xi \in  C^0([0,T];H^1_{\langle \xi\rangle_+^{\alpha_1}})\cap L^2([0,T];L^2_{\langle \xi\rangle_+^{\alpha_4}}),\right.\right.\\[6pt]
    &\qquad\qquad\qquad\qquad\left.\phi_{\xi\xi}\in L^2([0,T];H^1_{\langle \xi\rangle_+^{\alpha_3}}) \right\},
\end{array}
  \end{equation}
  where $\alpha_1=\frac{2}{1-m}$ and $\alpha_2=\frac{2m}{1-m}$, $\alpha_3=\frac{3-m}{1-m}$ and $\alpha_4=\frac{1+m}{1-m}$.\\

 \item [$\bullet$] for the degenerate case of $f'(u_+)=s< f'(u_-)$:
\begin{equation}
\begin{array}{lll}
  Y(0,T):=
    &\left\{\phi(\xi,t)\left|\phi \in C^0([0,T];L^2_{\langle \xi\rangle_+^{\beta_2}}),~\phi_\xi \in  C^0([0,T];H^1_{\langle \xi\rangle_+^{\beta_1}})\cap L^2([0,T];L^2_{\langle \xi\rangle_+}),\right.\right.\\[5pt]
    &\qquad\qquad\qquad\qquad\left.\phi_{\xi\xi}\in L^2([0,T];H^1_{\langle \xi\rangle_+^{\beta_3}}) \right\},
\end{array}
    \end{equation}
    where $\beta_1=\frac{2}{k_++1-m}$, $\beta_2=\frac{k_+}{k_++1-m}$ and $\beta_3=\frac{3-m}{k_++1-m}$.
  \end{itemize}

  Define a weight function $w$ as
  \begin{equation}
    w:=w(U)=\df{U(U-u_-)}{g(U)},~\mbox{for}~U\in[u_+,u_-].
  \end{equation}
We put
\begin{equation}\label{Nt1}
  N_1(T):=
  \mathop{\sup}\limits_{t\in[0,T]}\left\{\left\|\df{\phi_{\xi\xi}(t)}{U}\right\|^2+\left\|\df{\phi_{\xi}(t)}{U}\right\|^2
+ \left\|\df{\phi(t)}{U^m}\right\|^2\right\},
\end{equation}
and
\begin{equation}\label{Nt2}
  N_2(T):=
  \mathop{\sup}\limits_{t\in[0,T]}\left\{\left\|\df{\phi_{\xi\xi}(t)}{U}\right\|^2+\left\|\df{\phi_{\xi}(t)}{U}\right\|^2
+\left\|{\phi(t)}\right\|_w^2\right\}.
\end{equation}

The global smooth solution in each solution space will be  proved by a weighted energy method combining the local existence with the $a~ priori$ estimates. Since the local existence can be proved in the standard way, we are mainly concerned with the $a~priori$ estimates which will be proved in Section 5 and Section 6.

\begin{prop}[$A~priori$ estimates for the case of $f'(u_+)<s<f'(u_-)$]\label{prop}

Under  condition  \eqref{M},  let $\phi$ be a solution in $X(0,T)$ for a positive constant $T$. Then, there exists a positive constant $\delta$ such that if
\begin{equation}
  N_1(t)\leq \delta,
\end{equation}
then for $0\leq t\leq T$,  it holds that
 \begin{eqnarray}\label{final}
   && \|{\phi_{\xi\xi}(t)}\|_{\langle \xi\rangle_+^{\alpha_1}}^2
   +\|{\phi_{\xi}(t)}\|_{\langle \xi\rangle_+^{\alpha_1}}^2
   +\|{\phi(t)}\|_{\langle \xi\rangle_+^{\alpha_2}}^2 \\
   &&+\displaystyle{\int_0^t \Big(\|{\phi_{\xi\xi\xi}(\tau)}\|_{\langle \xi\rangle_+^{\alpha_3}}^2
   +\| {\phi_{\xi\xi}(\tau)}\|_{\langle \xi\rangle_+^{\alpha_3}}^2
   +\|{\phi_{\xi}(\tau)}\|_{\langle \xi\rangle_+^{\alpha_4}}^2\Big) \;{\rm d}\tau }\notag \\
   &&\leq CN_1(0), \notag
  \end{eqnarray}
  where $\alpha_1=\frac{2}{1-m}$, $\alpha_2=\frac{2m}{1-m}$, $\alpha_3=\frac{3-m}{1-m}$ and $\alpha_4=\frac{1+m}{1-m}$.
\end{prop}

Once the Proposition \ref{prop} is obtained, we can show the following global existence theorem, which implies Theorem \ref{Mthm} by defining $\phi_\xi=u-U$.

\begin{thm}[The case of $f'(u_+)<s<f'(u_-)$]\label{ph1}
Assume that \eqref{M} holds and $\phi_0\in L^2_{\langle \xi\rangle_+^{\alpha_2}}$ and $\phi_{0\xi}\in H_{\langle \xi\rangle_+^{\alpha_1}}^1$, where $\alpha_1=\frac{2}{1-m}$ and $\alpha_2=\frac{2m}{1-m}$. There exists a positive constant $\delta_0$ such that if $N_1(0)\leq \delta_0$, then the Cauchy problem for \eqref{Main} has a unique global solution $\phi\in X(0,+\infty)$. In addition, it satisfies
\begin{equation}
 \sup_{\xi\in R}|\phi_\xi(\xi,t)|\rightarrow0,~\mbox{as}~t\rightarrow+\infty.
\end{equation}

\end{thm}

\begin{prop}[$A~priori$ estimates for the case of  $f'(u_+)=s<f'(u_-)$]\label{prop1}
Let $\phi$ be a solution in $Y(0,T)$ for a positive constant $T$. Then, there exists a positive constant $\delta$ such that if
\begin{equation}
  N_2(t)\leq \delta,
\end{equation}
then for $0\leq t\leq T$,  it holds that
 \begin{equation}\label{finalde}
  \begin{array}{ll}
 \left\|{\phi_{\xi\xi}(t)}\right\|_{\langle \xi\rangle_+^{\beta_1}}^2&+\left\|{\phi_{\xi}(t)}\right\|_{\langle \xi\rangle_+^{\beta_1}}^2+\left\|{\phi(t)}\right\|_{\langle \xi\rangle_+^{\beta_2}}^2+\displaystyle\int_0^t \left(\left\|{\phi_{\xi\xi\xi}(\tau)}\right\|_{\langle \xi\rangle_+^{\beta_3}}^2+\left\| {\phi_{\xi\xi}(\tau)}\right\|_{\langle \xi\rangle_+^{\beta_3}}^2\right.\\[6pt]
   &\left.+\left\|{\phi_{\xi}(\tau)}\right\|_{\langle \xi\rangle_+}^2\right) \;{\rm d}\tau \leq CN_2(0),
  \end{array}
  \end{equation}
  where $\beta_1=\frac{2}{k_++1-m}$, $\beta_2=\frac{k_+}{k_++1-m}$and $\beta_3=\frac{3-m}{k_++1-m}$.
\end{prop}
Similarly, once the Proposition \ref{prop1} is obtained, we can show the following global existence theorem, which implies Theorem \ref{ext} by defining $\phi_\xi=u-U$.

\begin{thm}[The case of  $f'(u_+)=s<f'(u_-)$]\label{ph2}
Assume $\phi_0\in L^2_{\langle \xi\rangle_+^{\beta_2}}$ and $\phi_{0\xi}\in H_{\langle \xi\rangle_+^{\beta_1}}^1$, where $\beta_1=\frac{2}{k_++1-m}$ and $\beta_2=\frac{k_+}{k_++1-m}$.  There exists a positive constant $\delta_0$ such that if $N_2(0)\leq \delta_0$, then the Cauchy problem has a unique global solution $\phi\in Y(0,+\infty)$. In addition, it satisfies
\begin{equation}
 \sup_{\xi\in R}|\phi_\xi(\xi,t)|\rightarrow0,~\mbox{as}~t\rightarrow+\infty.
\end{equation}
\end{thm}

\section{Existence of shock profiles}

\begin{proof}[Proof of Theorem \ref{pro}]

We assume that \eqref{system} admits a viscous shock wave $U(\xi)$  connecting $u_-$ and $u_+$, where $\xi=x-st$. Integrating the first equation of \eqref{xi} in $\xi$ over
$R$  and noting $U_\xi(\pm \infty)=0$, we get
\begin{equation}
  -s(u_+-u_-)+f(u_+)-f(u_-)=0,
\end{equation}
which leads to the Rankine-Hugoniot condition \eqref{RH}. On the other hand, the ordinary differential equation \eqref{Hu} with $h(u_\pm)=0$ admits a unique smooth solution if and only if the following condition holds: When $u_+<u_-$,
\begin{equation}\label{hsma}
  h(u)<0,~ \mbox{for}~u\in (u_+, u_-),
\end{equation}
which gives us that $g(u)<0$ when $u\in(u_+,u_-)$. Thus in order to show monotonicity of the viscous shock wave, the necessary conditions are the Rankine-Hugoniot condition \eqref{RH} and the generalized entropy condition \eqref{Sc}.

Conversely, we suppose that \eqref{RH} and \eqref{Sc} hold, then we get $U(\xi)<0$ for $U\in (u_+,u_-)$. Let $u_*=\frac{u_++u_-}{2}$ and  $U(\xi_*)=u_*$,  then we have an ordinary differential equation
\begin{equation}
  \left\{
  \begin{array}{ll}
\dfrac{{\rm d}U}{{\rm d}h(U)}={\rm d} \xi,\\[10pt]
U(\xi_*)=u_*.
   \end{array}
    \right.
\end{equation}
Denote
\begin{equation}
  H(U):= \int_{u_*}^U \dfrac{1}{h(\tau)}{\rm d} \tau=\xi-\xi_*.
  \end{equation}
  Notice that
  \begin{equation}
    H'(U)=\dfrac{1}{h(U)}<0,
  \end{equation}
  which means that $H(U)$ is invertible. From $H(U)=\xi-\xi_*$, we get $U=H^{-1}(\xi)$ which is the desired shock profile.

Furthermore, we are going to show the convergence rates for $U(\xi)$ as $\xi\rightarrow\pm\infty$. For $f'(u_+)<s<f'(u_-)$, we have
\begin{equation}
  \dfrac{{\rm d}U}{{\rm d}\xi}\sim\dfrac {f'(u_+)-s}{\mu m}U^{2-m},~\mbox{as}~\xi \rightarrow+\infty.
\end{equation}
A straightforward computation yields that
\begin{equation}
  |U_\xi|=|h(U)|\sim |U|^{2-m}, \ \mbox{ as } \xi\rightarrow+\infty,
\end{equation}
and
\begin{equation}
  |U(\xi)|\sim|\xi|^{-\frac{1}{1-m}}, \ \mbox{ as } \xi\rightarrow+\infty.
\end{equation}
Similarly,  we have
\begin{equation}
  \dfrac{{\rm d}U}{{\rm d}\xi}\sim \dfrac {f'(u_-)-s}{\mu m}u_-^{1-m}(U-u_-)=\lambda (U-u_-),~\mbox{with}~\lambda=\dfrac {f'(u_-)-s}{\mu m}u_-^{1-m}, \ \mbox{ as } \xi\rightarrow -\infty.
\end{equation}
A direct calculation gives us that
\begin{equation}
  |U_\xi|=|h(U)|\sim |U-u_-|, \ \mbox{ as } \xi\rightarrow-\infty,
\end{equation}
and
\begin{equation}
  |U-u_-|\sim e^{-\lambda|\xi|}, \ \mbox{ as } \xi\rightarrow-\infty.
\end{equation}
Especially, when $s=f'(u_+)$ we have
\begin{equation}
  \dfrac{{\rm d}U}{{\rm d}\xi}\sim \df{f^{(k_++1)}(u_+)}{\mu m}(U-u_+)^{k_++1}U^{1-m},~\mbox{as}~\xi\rightarrow
+\infty,
\end{equation}
then we obtain that
\begin{equation}
  |U_\xi|=|h(U)|\sim |U|^{k_++2-m}, \ \mbox{ as } \xi\rightarrow +\infty,
\end{equation}
and
\begin{equation}
  |U(\xi)|\sim |\xi|^{-\frac{1}{k_++1-m}}, \ \mbox{ as } \xi\rightarrow +\infty.
\end{equation}

Thus, the proof is completed.
\end{proof}

\section{${A~priori}$ estimates for the case of $f'(u_+)<s<f'(u_-)$}

In what follows, we confine ourselves to the proof of Proposition \ref{prop} by establishing the $a~priori$~estimates for the case $f'(u_+)<s<f'(u_-)$.
Let $\phi(\xi,t)\in X(0,T) $ be the solution of \eqref{Main}-\eqref{G1} with
\begin{equation}\label{pri1}
  N_1(t)\leq \delta,~\mbox{for}~t\in[0,T],
\end{equation}
which implies that
\begin{equation}\label{inf}
\sup_{t\in[0,T]}\left(\left\|\df{\phi(t)}{U^m}\right\|_{L^\infty} +\left\|\df{\phi_\xi(t)}{U}\right\|_{L^\infty}\right)\leq C\delta.
\end{equation}

\begin{lem}\label{uxi1}
  In the case of $f'(u_+)<s<f'(u_-)$, it holds that
  \begin{equation}
    |U_\xi|\leq CU^{2-m},
  \end{equation}
  and
  \begin{equation}
    |U_{\xi\xi}|\leq CU^{3-2m},
  \end{equation}
  for all $\xi\in{\rm R}$.
\end{lem}
\begin{proof}
  This lemma is a direct corollary of Theorem \ref{pro}.
\end{proof}

\begin{lem}\label{F}
Under the $a~priori$ assumption \eqref{pri1}, it holds that
\begin{equation}
|F|\leq C\phi_\xi^2,
\end{equation}
and
\begin{equation}
|F_\xi|\leq C(U^{2-m}|\phi_\xi|^2+|\phi_\xi||\phi_{\xi\xi}|).
 \end{equation}

\end{lem}

\begin{proof}
From the representation of $F=-\{f(U+\phi_\xi)-f(U)-f'(U)\phi_\xi\}$ and Taylor's expansion, a direct calculation gives
\begin{equation}
\begin{split}
|F_\xi|&=|\left\{f(U+\phi_\xi)-f(U)-f'(U)\phi_\xi\right\}_\xi|\\
&=|\left(f'(U+\phi_\xi)-f'(U)-f''(U)\phi_\xi\right)U_\xi+\left(f'(U+\phi_\xi)-f'(U)\right)\phi_{\xi\xi}|\\&\leq |f'(U+\phi_\xi)-f'(U)-f''(U)\phi_\xi||U_\xi|+|f'(U+\phi_\xi)-f'(U)||\phi_{\xi\xi}|\\
&\leq C(|\phi_\xi|^2|U_\xi|+|\phi_\xi||\phi_{\xi\xi}|).
 \end{split}
 \end{equation}
which implies Lemma \ref{F} combining with \eqref{inf} and Lemma \ref{uxi1}.
\end{proof}

\begin{lem}\label{G}
Under the $a~priori$ assumption \eqref{pri1}, it holds that
\begin{equation}
|G|\leq \dfrac{C\phi_\xi^2}{U^{2-m}},
\end{equation}
\begin{equation}
|G_\xi|\leq 
\left\{\dfrac{|\phi_\xi|^2}{U}+\dfrac{|\phi_\xi||\phi_{\xi\xi}|}{U^{2-m}}\right\},
\end{equation}
and
\begin{equation}
\begin{split}
|G_{\xi\xi}|\leq C 
\left\{\dfrac{|\phi_{\xi\xi\xi}||\phi_\xi|}{U^{2-m}}+\dfrac{|\phi_\xi|^2}{U^{m}}+\dfrac{|\phi_{\xi\xi}||\phi_\xi|}{U}+\dfrac{|\phi_{\xi\xi}|^2}{U^{2-m}}\right\}.
\end{split}
 \end{equation}
\end{lem}

\begin{proof}
From the representation of $G=\mu\{(U+\phi_\xi)^m-U^m- m\dfrac{\phi_\xi}{U^{1-m}}\}$, by a direct computation, we have
\begin{equation}
\begin{split}
 \left| G_\xi\right|&=\left|\mu m\left\{U_\xi\left(\dfrac{1}{(U+\phi_\xi)^{1-m}}-\dfrac{1}{U^{1-m}}+(1-m)\dfrac{\phi_\xi}{U^{2-m}}\right)+\phi_{\xi\xi}\left(\dfrac{1}{(U+\phi_\xi)^{1-m}}-\dfrac{1}{U^{1-m}}\right)\right\}\right|\\
&\leq C \left\{\left|U_\xi\right|\left|\dfrac{1}{(U+\phi_\xi)^{1-m}}-\dfrac{1}{U^{1-m}}+(1-m)\dfrac{\phi_\xi}{U^{2-m}}\right|+|\phi_{\xi\xi}|\left|\dfrac{1}{(U+\phi_\xi)^{1-m}}-\dfrac{1}{U^{1-m}}\right|\right\}\\[4pt]
&\leq C \left\{\df{|U_\xi||\phi_{\xi}|^2}{U^{3-m}}+\df{|\phi_{\xi\xi}||\phi_\xi|}{U^{2-m}}\right\},
\end{split}
\end{equation}
and
 \begin{equation}
\begin{split}
\left|G_{\xi\xi}\right|
&=\mu m\left|\left\{U_{\xi\xi}\left(\dfrac{1}{(U+\phi_\xi)^{1-m}}-\dfrac{1}{U^{1-m}}+(1-m)\dfrac{\phi_\xi}{U^{2-m}}\right)
+\phi_{\xi\xi\xi}\left(\dfrac{1}{(U+\phi_\xi)^{1-m}}-\dfrac{1}{U^{1-m}}\right)\right.\right.\\
&\quad-(1-m)U_\xi^2\left(\dfrac{1}{(U+\phi_\xi)^{2-m}}-\dfrac{1}{U^{2-m}}+(2-m)\dfrac{\phi_\xi}{U^{3-m}}\right)\\
&\quad-2(1-m)U_\xi\phi_{\xi\xi}\left(\dfrac{1}{(U+\phi_\xi)^{2-m}}-\dfrac{1}{U^{2-m}}\right)
-(1-m)\left.\left.\dfrac{\phi_{\xi\xi}^2}{(U+\phi_\xi)^{2-m}}\right\}\right|\\[4pt]
&\leq C\left\{ \dfrac{|U_{\xi\xi}||\phi_\xi|^2}{U^{3-m}}+\dfrac{|\phi_{\xi\xi\xi}||\phi_\xi|}{U^{2-m}}
+\dfrac{|U_\xi|^2|\phi_\xi|^2}{U^{4-m}}+\dfrac{|U_\xi||\phi_{\xi\xi}||\phi_\xi|}{U^{3-m}}+\dfrac{|\phi_{\xi\xi}|^2}{U^{2-m}}\right\}.
 \end{split}
 \end{equation}
By virtue of \eqref{inf} and Theorem \ref{uxi1}, the proof of Lemma \ref{G} is completed.
\end{proof}

Let us introduce a cut-off function  $\eta(\xi)$, which will be applied to treat the singularities of solution at the far fields $\xi=\pm\infty$.

\begin{lem}\label{cutoff}
  Given $L>0$, define the so-called cut-off function $\eta{(\xi)}$ by
  \begin{equation}\label{cutoff-1}
    \eta{(\xi)}:=\left\{
    \begin{array}{lll}
    1,~&\mbox{for}~0\leq|\xi|<L,\\
      \exp\left(1-\frac{L}{2L-|\xi|}\right),~&\mbox{for}~L\leq|\xi|<2L,\\
      0,&\mbox{for}~|\xi|\geq 2L.
    \end{array}\right.
  \end{equation}
  Then, it holds that
  \begin{equation}
    |\eta_\xi(\xi)|\leq \df{C}{L},~\mbox{for}~\xi\in{\rm R}.
  \end{equation}
\end{lem}

 \begin{proof} The proof is straightforward (see also \cite{cut1,cut2}).
   We omit its details.
 \end{proof}

\begin{lem}\label{l1}
Under the same conditions as those in Proposition \ref{prop}, it holds that
\begin{equation}\label{L1}
\left\|\df{\phi(t)}{U^{m}}\right\|^2+\int_0^t\int_R \dfrac{\phi_\xi^2}{U^{1+m}}~{\rm d}\xi{\rm d}\tau \leq C\left\|\df{\phi(0)}{U^{m}}\right\|^2,
\end{equation}
for $t\in[0,T]$ provided $\delta\ll 1$. Thus, it holds that
\begin{equation}\label{L11}
\|\phi(t)\|_{\langle \xi\rangle_+^{\alpha_2}}^2+\int_0^t\int_0^t\|\phi_\xi(\tau)\|_{\langle \xi\rangle_+^{\alpha_4}}^2\,{\rm d}\tau\leq C\|\phi(0)\|_{\langle \xi\rangle_+^{\alpha_2}}^2,
\end{equation}
with $\alpha_2={\frac{2m}{1-m}}$ and $\alpha_4={\frac{1+m}{1-m}}$ for $t\in[0,T]$ provided $\delta\ll 1$.
\end{lem}

\begin{proof}
Multiplying \eqref{Main} by $\dfrac{\phi}{U^{2m}}$, we obtain
\begin{equation}\label{l127}
\begin{array}{lll}
  &\left(\df{1}{2}\dfrac{\phi^2}{U^{2m}}\right)_t+\left(\dfrac{1}{2}g'(U)\dfrac{\phi^2}{U^{2m}} -\mu m\dfrac{\phi\phi_\xi}{U^{1+m}}-\mu m^2\dfrac{\phi^2U_\xi}{U^{2+m}}-G\dfrac{\phi}{U^{2m}}\right)_\xi\\[12pt]
&-\df{1}{2}\left(\df{g(U)}{U^{2m}}\right)''U_\xi\phi^2
+\mu m\dfrac{\phi_\xi^2}{U^{1+m}}=F\dfrac{\phi}{U^{2m}}-G\left(\dfrac{\phi_\xi}{U^{2m}}-2m\dfrac{U_\xi\phi}{U^{2m+1}}\right).
\end{array}
\end{equation}
Note that, when we  integrate \eqref{l127} over $(-\infty,\infty)$, the second term of \eqref{l127} will not disappear at the far fields $x=\pm\infty$, and possesses some singularities. This is an obstacle. In order to treat it, we are going to apply the cut-off function technique as follows.
 
Multiplying \eqref{l127} by the cut-off function $\eta(\xi)$ introduced in \eqref{cutoff-1},  integrating the resultant  equation over ${\rm R}$, and noting $K(U)=\frac{g(U)}{U^{2m}}$, we get
\begin{equation}\label{27}
\begin{array}{lll}
&\df{{\rm d}}{{\rm d}t}\displaystyle \int_{-2L}^{2L} \dfrac{\phi(t)^2\eta}{U^{2m}}{\rm d}\xi+\displaystyle \int_{-2L}^{2L}-\df{1}{2}K''(U)U_\xi\phi^2\eta{\rm d}\xi+\displaystyle \int_{-2L}^{2L}\mu m\dfrac{\phi_\xi^2\eta}{U^{1+m}}{\rm d}\xi\\[14pt]
&=\D\int_{-2L}^{2L} \dfrac{F\phi\eta}{U^{2m}}-G\left(\dfrac{\phi_\xi}{U^{2m}}-2m\dfrac{U_\xi\phi}{U^{2m+1}}\right)\eta{\rm d}\xi\\[14pt]
&\D+\int_{-2L}^{2L}\eta_\xi(\xi)\left(\dfrac{1}{2}g'(U)\dfrac{\phi^2}{U^{2m}} -\mu m\dfrac{\phi\phi_\xi}{U^{1+m}}-\mu m^2\dfrac{\phi^2U_\xi}{U^{2+m}}-G\dfrac{\phi}{U^{2m}}\right){\rm d}\xi.
\end{array}
\end{equation}
Given the convexity of $K(u)$,  for the second term on the left hand side of \eqref{27},  we have
\begin{equation}
 \int_{-2L}^{2L}-\df{1}{2}K''(U)U_\xi\phi^2\eta~{\rm d}\xi\geq C\int_{-2L}^{2L}|K''(U)||U_\xi|\phi^2\eta~{\rm d}\xi.
\end{equation}
 As for the last term on the right hand side of \eqref{27}, let us first state the following corollary:  $\frac{\phi}{U^{m}}\in L^2_{loc}({\rm R})$
and
  $\frac{\phi_\xi}{U}\in L^2_{loc}({\rm R})$
  for $t\in[0,T]$. Consequently, it can be inferred directly from \eqref{Nt1} and \eqref{pri1}.  Furthermore, we have
  \begin{equation}
    \int_{-2L}^{2L}\df{\phi^2}{U^{2m}}+\df{\phi^2}{U^2}\,{\rm d}\xi\leq  \int_R\df{\phi^2}{U^{2m}}+\df{\phi^2}{U^2}\,{\rm d}\xi\leq N_1(t)\leq \delta.
  \end{equation}
  Then, we have
\begin{equation}
  \int_{-2L}^{2L}\dfrac{1}{2}\eta_\xi(\xi)g'(U)\dfrac{\phi^2}{U^{2m}} {\rm d}\xi\leq C\int_{-2L}^{2L}\dfrac{|\eta_\xi(\xi)|\phi^2}{U^{2m}} {\rm d}\xi\leq \df{C}{L}\int_{-2L}^{2L}\dfrac{\phi^2}{U^{2m}} {\rm d}\xi\leq\df{C\delta}{L}.
\end{equation}
Similarly, applying the Cauchy-Schwarz inequality, we have
\begin{equation}
\int_{-2L}^{2L}-\mu m\eta_\xi(\xi)\dfrac{\phi\phi_\xi}{U^{1+m}}{\rm d}\xi\leq C\int_{-2L}^{2L}\dfrac{|\eta_\xi(\xi)||\phi||\phi_\xi|}{U^{1+m}}{\rm d}\xi\leq \df{C}{L} \int_{-2L}^{2L}\dfrac{\phi^2}{U^{2m}}+\dfrac{\phi_\xi^2}{U^2 }{\rm d}\xi\leq\df{C\delta}{L}.
\end{equation}
By virtue of Lemma \ref{uxi1}, we get that $|U_\xi|\leq CU^{2-m}$ for $U\in[u_+,u_-]$. Thus, we obtain
\begin{equation}
\int_{-2L}^{2L}-\mu m^2\dfrac{\eta_\xi(\xi)\phi^2U_\xi}{U^{2+m}}{\rm d}\xi\leq C\int_{-2L}^{2L}\dfrac{|\eta_\xi(\xi)|\phi^2|U_\xi|}{U^{2+m}}{\rm d}\xi\leq \df{C}{L} \int_{-2L}^{2L}\dfrac{\phi^2}{U^{2m}}{\rm d}\xi\leq\df{C\delta}{L}.
\end{equation}
Using \eqref{inf}, Lemma \ref{G} and the Cauchy-Schwarz inequality, we have
\begin{equation}\label{l134}
\begin{array}{lll}
  &\D\int_{-2L}^{2L}-G\eta_\xi(\xi)\dfrac{\phi}{U^{2m}}{\rm d}\xi\\[14pt]
  &\D\leq C\int_{-2L}^{2L}\dfrac{|\eta_\xi(\xi)||\phi||\phi_\xi|^2}{U^{2+m}}{\rm d}\xi\\[14pt]
  &\D\leq \df{C}{L}\left \|\df{\phi(t)}{U^m}\right\|_{L^\infty}\int_{-2L}^{2L}\dfrac{\phi_\xi^2}{U^2 }{\rm d}\xi\\[14pt]
  &\D\leq\df{C\delta^2}{L}.
\end{array}
\end{equation}
Combining \eqref{27}-\eqref{l134}, we get
\begin{equation}
  \begin{array}{lll}
&\df{{\rm d}}{{\rm d}t}\displaystyle \int_{-2L}^{2L} \dfrac{\phi(t)^2\eta}{U^{2m}}{\rm d}\xi+\int_{-2L}^{2L}|K''(U)||U_\xi|\phi^2\eta{\rm d}\xi+\displaystyle \int_{-2L}^{2L}\dfrac{\phi_\xi^2\eta}{U^{1+m}}{\rm d}\xi\\[14pt]
&\leq\D C\left(\int_{-2L}^{2L} \dfrac{F\phi\eta}{U^{2m}}-G\left(\dfrac{\phi_\xi}{U^{2m}}-2m\dfrac{U_\xi\phi}{U^{2m+1}}\right)\eta\,{\rm d}\xi+\df{\delta
+\delta^2}{L}\right).
\end{array}
\end{equation}
Integrating the above equation over $[0,t]$ and taking $L\rightarrow+\infty$, we obtain
\begin{equation}\label{l137}
\begin{array}{ll}
&\D\left\|\df{\phi(t)}{U^{m}}\right\|^2+\int_0^t\int_R \dfrac{\phi_\xi^2}{U^{1+m}}~{\rm d}\xi{\rm d}\tau \\[14pt]
&\leq C\left(\D\left\|\df{\phi(0)}{U^{m}}\right\|^2+\int_0^t\int_R
 \dfrac{F\phi}{U^{2m}}-G\left(\dfrac{\phi_\xi}{U^{2m}}-2m\dfrac{U_\xi\phi}{U^{2m+1}}\right)~{\rm d}\xi{\rm d}\tau\right).
\end{array}
\end{equation}
For the last term on the right hand side of \eqref{l137}, by virtue of \eqref{inf} and Lemma  \ref{F}, we have
\begin{equation}\label{l138}
\begin{split}
    &\D \int_0^t\int_R \dfrac{F\phi}{U^{2m}}{\rm d}\xi\,{\rm d}\tau\\[5pt]
    &\D\leq C\int_0^t\int_R \df{|\phi||\phi_\xi|^2}{U^{2m}}\,{\rm d}\xi{\rm d}\tau\\[5pt]
     &\D\leq C  \int_0^t \left \|{\phi}\right\|_{L^\infty} \int_R \df{|\phi_\xi|^2}{U^{2m}}\,{\rm d}\xi{\rm d}\tau\\[5pt]
    &\D\leq C\delta \int_0^t\int_R \df{|\phi_\xi|^2}{U^{1+m}}\,{\rm d}\xi{\rm d}\tau.
\end{split}
\end{equation}
Similarly, taking into consideration  \eqref{inf}, Lemma \ref{uxi1} and Lemma \ref{G}, we get
\begin{equation}\label{l139}
\begin{split}
 &\int_0^t\int_R -G\left(\dfrac{\phi_\xi}{U^{2m}}-2m\dfrac{U_\xi\phi}{U^{2m+1}}\right)\,{\rm d}\xi{\rm d}\tau \\[5pt]
 &\leq C\int_0^t\int_R\dfrac{|\phi_\xi|^3}{U^{2+m}}
+\dfrac{|\phi_\xi|^2|\phi||U_\xi|}{U^{3+m}}\,{\rm d}\xi{\rm d}\tau\\[5pt]
&\leq C\left(\int_0^t\left\|\df{\phi_\xi}{U}\right\|_{L^\infty}\int_R\dfrac{|\phi_\xi|^2}{U^{1+m}}\,{\rm d}\xi{\rm d}\tau+\int_0^t\left\|\df{\phi}{U^m}\right\|_{L^\infty}\int_R\dfrac{|\phi_\xi|^2}{U^{1+m}}\,{\rm d}\xi{\rm d}\tau  \right)\\[5pt]
&\leq C\delta \int_0^t\int_R\dfrac{|\phi_\xi|^2}{U^{1+m}}\,{\rm d}\xi{\rm d}\tau.
\end{split}
\end{equation}
Substituting \eqref{l138} and \eqref{l139} into \eqref{l137} and taking $\delta$ sufficiently small, we get \eqref{L1}.

When $f'(u_+)<s<f'(u_-)$, we have
\begin{equation}
U(\xi)\sim\left\{
  \begin{array}{ll}
    |\xi|^{-\frac{1}{1-m}},&~\mbox{as}~\xi\rightarrow+\infty,\\
1,&~\mbox{as}~\xi\rightarrow-\infty,
  \end{array}
  \right.
  \end{equation}
which implies that $\frac{1}{U^{2m}}\sim {\langle \xi\rangle_+^{\alpha_2}}$ and  $\frac{1}{U^{1+m}}\sim {\langle \xi\rangle_+^{\alpha_4}}$ for $\xi\in {\rm R}$. We derive \eqref{L11}.
\end{proof}

Next, we estimate $\phi_\xi$.

\begin{lem}\label{h1}
Under the same conditions as those in Proposition \ref{prop}, it holds that
\begin{equation}\label{H11}
\left\|\df{{\phi_\xi(t)}}{U}\right\|^2+\int_0^t\int_R \dfrac{\phi_{\xi\xi}^2}{U^{3-m}} \,{\rm d}\xi{\rm d}\tau \leq C\left\|\df{{\phi_\xi(0)}}{U}\right\|^2
+ C\left\|\df{{\phi(0)}}{U^{m}}\right\|^2,
\end{equation}
for $t\in[0,T]$ provided $\delta\ll 1$. Thus, it holds that
\begin{equation}\label{h11}
\|\phi_\xi(t)\|_{\langle \xi\rangle_+^{\alpha_1}}^2+\int_0^t\int_0^t\|\phi_{\xi\xi}(\tau)\|_{\langle \xi\rangle_+^{\alpha_3}}^2\,{\rm d}\tau\leq C\|\phi_\xi(0)\|_{\langle \xi\rangle_+^{\alpha_1}}^2+C\|\phi(0)\|_{\langle \xi\rangle_+^{\alpha_2}}^2,
\end{equation}
with $\alpha_1={\frac{2}{1-m}}$, $\alpha_2={\frac{2m}{1-m}}$ and $\alpha_3={\frac{3-m}{1-m}}$ for $t\in[0,T]$ provided $\delta\ll 1$.
\end{lem}

\begin{proof}

To prevent singularity in the weight function, we define $U_\varepsilon=U+\varepsilon$, where $\varepsilon$ is a positive constant. Multiplying \eqref{Main} by $-\left(\dfrac{\phi_\xi}{U^2_\varepsilon}\right)_\xi$, we get
\begin{equation}\label{h1s}
\begin{split}
 & \left(\dfrac{1}{2}\dfrac{\phi^2_\xi}{U^2_\varepsilon}\right)_t -\left(\df{\phi_t\phi_\xi}{U^2_\varepsilon}+g'(U)\dfrac{\phi^2_\xi}{U^2_\varepsilon}-\df{\phi_\xi}{U_\varepsilon^2}F\right)_\xi+\left(g''(U)\dfrac{U_\xi}{U_\varepsilon^2}+2\mu m(1-m)\dfrac{U^2_\xi}{U^{2-m}U_\varepsilon^3}\right)\phi_\xi^2\\
  &+\mu m \dfrac{1}{U^{1-m}U^2_\varepsilon}\phi^2_{\xi\xi}
  +\left(\dfrac{g'(U)}{U_\varepsilon^2}-\dfrac{\mu m(1-m)U_\xi}{U^{2-m}U^2_\varepsilon}
-\dfrac{2\mu mU_{\xi}}{U^{1-m}U_\varepsilon^3}\right)\phi_\xi\phi_{\xi\xi}\\
&=\dfrac{\phi_\xi}{U^2_\varepsilon} F_\xi-\left(\dfrac{\phi_\xi}{U^2_\varepsilon}\right)_\xi G_{\xi}.
\end{split}
\end{equation}
Let us denote
\begin{equation}
\begin{array}{ll}
  Q_1(U):=g''(U)\dfrac{U_\xi}{U_\varepsilon^2}+2\mu m(1-m)\dfrac{U^2_\xi}{U^{2-m}U_\varepsilon^3},
\end{array}
\end{equation}
and
\begin{equation}
  Q_2(U):=\dfrac{g'(U)}{U_\varepsilon^2}-\dfrac{\mu m(1-m)U_\xi}{U^{2-m}U^2_\varepsilon}
-\dfrac{2\mu mU_{\xi}}{U^{1-m}U_\varepsilon^3}.
\end{equation}
By virtue of Lemma \ref{uxi1}, we have $|U_\xi|\leq CU^{2-m}$ for $U\in[u_+,u_-]$, which implies
\begin{equation}
 | Q_1(U)|\leq C\left(\df{|U_\xi|}{U^2}+\df{|U_\xi|^2}{U^{5-m}}\right) \leq C\left(\df{1}{U^m}+\df{1}{U^{1+m}}\right)\leq\df{C}{U^{1+m}},~\mbox{for}~U\in[u_+,u_-].
\end{equation}
Similarly, by the Cauchy-Schwarz inequality for any $\sigma>0$, we have
\begin{equation}
  \left|Q_2(U)\phi_\xi\phi_{\xi\xi}\right|\leq C\df{|\phi_\xi\phi_{\xi\xi}|}{U^2_\varepsilon}\leq C\left(\sigma \dfrac{\phi^2_{\xi\xi}}{U^{1-m}U_\varepsilon^2}+\dfrac{1}{\sigma}\dfrac{\phi^2_{\xi}}{U^{1+m}}\right),~\mbox{for}~U\in[u_+,u_-].
\end{equation}
Integrating \eqref{h1s} over ${\rm R}\times [0,t]$, we get
\begin{equation}\label{h1s6}
\begin{array}{ll}
  &\D\int_R \dfrac{\phi^2_\xi(t)}{U^2_\varepsilon}\,{\rm d}\xi+\int_0^t\int_R \dfrac{\phi_{\xi\xi}^2}{U^{1-m}U_\varepsilon^2} \,{\rm d}\xi{\rm d}\tau\leq C\int_R \dfrac{\phi^2_\xi(0)}{U^2_\varepsilon}\,{\rm d}\xi +C\sigma \int_0^t\int_R\dfrac{\phi^2_{\xi\xi}}{U^{1-m}U_\varepsilon^2}\,{\rm d}\xi{\rm d}\tau\\[14pt]
&\D+C\left(\dfrac{1}{\sigma}+1\right)\int_0^t\int_R\dfrac{\phi^2_{\xi}}{U^{1+m}}\,{\rm d}\xi{\rm d}\tau\D+C\int_0^t\int_R\dfrac{\phi_\xi}{U^2_\varepsilon} F_\xi\,{\rm d}\xi{\rm d}\tau-C\int_0^t\int_R\left(\dfrac{\phi_\xi}{U^2_\varepsilon}\right)_\xi G_{\xi}\,{\rm d}\xi{\rm d}\tau.
\end{array}
\end{equation}
Next,  our attention shifts to  estimating  the last two terms on the right hand side of \eqref{h1s6}. Taking into account \eqref{inf}, Lemmas \ref{uxi1}-\ref{G} and  applying the Cauchy-Schwarz inequality for any $\sigma>0$, we have
\begin{equation}
   \begin{array}{lll}
    &\D \int_0^t\int_R\dfrac{\phi_\xi}{U^2_\varepsilon} F_\xi\,{\rm d}\xi{\rm d}\tau\\[14pt]
    &\D\leq C\int_0^t\int_R\df{|\phi_\xi|^3}{U^{m}}+\df{|\phi_\xi|^2|\phi_{\xi\xi}|}{U^2_\varepsilon}\,{\rm d}\xi{\rm d}\tau\\[14pt]
    &\D\leq C  \int_0^t \left\|\df{\phi_\xi}{U}\right\|_{L^\infty}\int_RU^{1-m}{\phi_\xi^2}+\df{|\phi_\xi||\phi_{\xi\xi}|}{U_\varepsilon}\,{\rm d}\xi{\rm d}\tau\\[14pt]
    &\D\leq C\delta\int_0^t\int_R\left(\dfrac{1}{\sigma}+1\right)\dfrac{\phi_\xi^2}{U^{1+m}}+\sigma\df{\phi_{\xi\xi}^2}{U^{1-m}U^2_\varepsilon}\,{\rm d}\xi{\rm d}\tau,
   \end{array}
 \end{equation}
 and
 \begin{equation}
   \begin{array}{llll}
       &\D\int_0^t\int_R-\left(\dfrac{\phi_\xi}{U_\varepsilon^2}\right)_\xi G_{\xi}\,{\rm d}\xi{\rm d}\tau\\[14pt]
&\D\leq C\int_0^t\int_R \left(\dfrac{|\phi_{\xi\xi}|}{U^2_\varepsilon}+\dfrac{|\phi_\xi(t)|| U_{\xi}|}{U^{3}_\varepsilon}\right)\left(\dfrac{|\phi_\xi|^2}{U}+\dfrac{|\phi_\xi||\phi_{\xi\xi}|}{U^{2-m}}
\right)\,{\rm d}\xi{\rm d}\tau\\[14pt]
&\D\leq C \int_0^t\left\|\df{\phi_\xi}{U}\right\|_{L^\infty} \int_R \dfrac{|\phi_\xi||\phi_{\xi\xi}|}{U_\varepsilon^2}
+\dfrac{|\phi_\xi|^2}{U^{1+m}}
+\dfrac{ |\phi_{\xi\xi}|^2}{U^{1-m}U^2_\varepsilon}\,{\rm d}\xi{\rm d}\tau\\[14pt]
&\D\leq C\delta\int_0^t\int_R\left(\dfrac{1}{\sigma}+1\right)\dfrac{\phi_\xi^2}{U^{1+m}}+(\sigma+1)\df{\phi_{\xi\xi}^2}{U^{1-m}U^2_\varepsilon}\,{\rm d}\xi{\rm d}\tau.
   \end{array}
 \end{equation}
 Consequently, choosing  $\delta$ and $\sigma$ appropriately small, we get
\begin{equation}
\int_R \dfrac{\phi^2_\xi(t)}{U^2_\varepsilon}\,{\rm d}\xi+\int_0^t\int_R \dfrac{\phi_{\xi\xi}^2}{U^{1-m}U_\varepsilon^2} \,{\rm d}\xi{\rm d}\tau \leq C\left\|\dfrac{\phi_\xi(0)}{U}\right\|^2+C\int_0^t\int_R \dfrac{\phi_\xi^2}{U^{1+m}} \,{\rm d}\xi{\rm d}\tau.
\end{equation}
Thanks to Lemma \ref{l1}, we have
\begin{equation}
  \int_0^t\int_R \dfrac{\phi_\xi^2}{U^{1+m}} \,{\rm d}\xi{\rm d}\tau\leq C\left\|\df{\phi(0)}{U^{m}}\right\|^2.
\end{equation}
Taking $\varepsilon\rightarrow 0$ and applying Fatou's lemma, we deduce that
\begin{equation}
\begin{array}{lll}
&\displaystyle \varliminf\limits_{\varepsilon\rightarrow0}\left(\int_R \dfrac{\phi^2_\xi(t)}{U^2_\varepsilon}\,{\rm d}\xi+\int_0^t\int_R \dfrac{\phi_{\xi\xi}^2}{U^{1-m}U_\varepsilon^2} \,{\rm d}\xi{\rm d}\tau\right)\\[14pt]
   &\geq\displaystyle\int_R \varliminf\limits_{\varepsilon\rightarrow0}\dfrac{\phi^2_\xi(t)}{U^2_\varepsilon}\,{\rm d}\xi+\displaystyle\int_0^t\int_R \varliminf\limits_{\varepsilon\rightarrow0}\dfrac{\phi_{\xi\xi}^2}{U^{1-m}U_\varepsilon^2} \,{\rm d}\xi{\rm d}\tau\\[12pt]
   &=\displaystyle\int_R \dfrac{\phi^2_\xi(t)}{U^2}\,{\rm d}\xi+\int_0^t\int_R \dfrac{\phi_{\xi\xi}^2}{U^{3-m}} \,{\rm d}\xi{\rm d}\tau .
\end{array}
\end{equation}
Thus, we get \eqref{H11}.

When $f'(u_+)<s<f'(u_-)$, since
\begin{equation}
U(\xi)\sim\left\{
  \begin{array}{ll}
    |\xi|^{-\frac{1}{1-m}},&~\mbox{as}~\xi\rightarrow+\infty,\\
1,&~\mbox{as}~\xi\rightarrow-\infty,
  \end{array}
  \right.
  \end{equation}
  we similarly have
 $\frac{1}{U^{2}}\sim {\langle \xi\rangle_+^{\alpha_1}}$ and $\frac{1}{U^{3-m}}\sim {\langle \xi\rangle_+^{\alpha_3}}$  for $\xi\in {\rm R}$. We get \eqref{h11} and  complete the proof of Lemma \ref{h1}.
\end{proof}

\begin{lem}\label{h2}
Under the same conditions as those in Proposition \ref{prop}, it holds that
\begin{equation}\label{H2}
\left\|\df{\phi_{\xi\xi}(t)}{U}\right\|^2+\int_0^t\int_R \dfrac{\phi_{\xi\xi\xi}^2}{U^{3-m}} \,{\rm d}\xi{\rm d}\tau \leq C\left\|{\df{\phi_{\xi}(0)}{U}}\right\|_{1}^2
+ C\left\|\df{{\phi(0)}}{U^m}\right\|^2,
\end{equation}
for $t\in[0,T]$ provided $\delta\ll 1$. Thus, it holds that
\begin{equation}\label{h21}
\|\phi_{\xi\xi}(t)\|_{\langle \xi\rangle_+^{\alpha_1}}^2+\int_0^t\int_0^t\|\phi_{\xi\xi\xi}(\tau)\|_{\langle \xi\rangle_+^{\alpha_3}}^2\,{\rm d}\tau\leq C\|\phi_\xi(0)\|_{1,\langle \xi\rangle_+^{\alpha_1}}^2+C\|\phi(0)\|_{\langle \xi\rangle_+^{\alpha_2}}^2,
\end{equation}
with $\alpha_1={\frac{2}{1-m}}$, $\alpha_2={\frac{2m}{1-m}}$ and $\alpha_3={\frac{3-m}{1-m}}$ for $t\in[0,T]$ provided $\delta\ll 1$.
\end{lem}

\begin{proof}

Differentiating \eqref{Main} with respect to $\xi$, multiplying it by $-\left(\dfrac{\phi_{\xi\xi}}{U_\varepsilon^{2}}\right)_\xi$ and integrating the resultant equation over ${\rm R}\times [0,t]$, we obtain
\begin{equation}\label{32}
  \begin{array}{llll}
   & \D\int_R \dfrac{\phi^2_{\xi\xi}(t)}{U^2_\varepsilon}\,{\rm d}\xi+\int_0^t\int_R \dfrac{\phi_{\xi\xi\xi}^2}{U^{1-m}U_\varepsilon^2} \,{\rm d}\xi{\rm d}\tau \\[14pt]
&\D\leq C\left\|\dfrac{\phi_{\xi\xi}(0)}{U}\right\|^2+C\int_0^t\int_R |E_1(U)|\phi_\xi^2 \,{\rm d}\xi{\rm d}\tau
+C\int_0^t\int_R |E_2(U)||\phi_\xi||\phi_{\xi\xi}| \,{\rm d}\xi{\rm d}\tau\\[14pt]
&\,\quad+\D
C\int_0^t\int_R |E_3(U)||\phi_{\xi\xi\xi}||\phi_{\xi\xi}|\ \,{\rm d}\xi{\rm d}\tau
+C\int_0^t\int_R |E_4(U)||\phi_\xi||\phi_{\xi\xi\xi}|\,{\rm d}\xi{\rm d}\tau \\[14pt]
&\quad\,\D+C\int_0^t\int_R\left(-\dfrac{\phi_{\xi\xi}}{U^2_\varepsilon}\right)_\xi F_\xi\,{\rm d}\xi{\rm d}\tau+C\int_0^t\int_R\left(-\dfrac{\phi_{\xi\xi}}{U^2_\varepsilon}\right)_\xi G_{\xi\xi}\,{\rm d}\xi{\rm d}\tau,
  \end{array}
\end{equation}
where
\begin{equation}
  E_1(U):=\dfrac{2g''(U)U_\xi}{U_\varepsilon^2}+\dfrac{4\mu m(1-m)U_\xi^2}{U^{2-m}U^3_\varepsilon},
\end{equation}
\begin{equation}
  E_2(U):=\dfrac{\left(g'''(U)U_\xi^2+g''(U)U_{\xi\xi}\right)}{U_\varepsilon^2}
-\mu m(1-m)\left(\dfrac{(2-m)U_\xi^2}{U^{3-m}}-\dfrac{U_{\xi\xi}}{U^{2-m}}\right)\dfrac{2U_\xi}{U^3_\varepsilon},
\end{equation}
\begin{equation}
  E_3(U):=\dfrac{g'(U)}{U^2_\varepsilon}-\dfrac{2\mu mU_\xi}{U^{1-m}U_\varepsilon^3}-\dfrac{2\mu m(1-m)U_\xi}{U^{2-m}U_\varepsilon^2},
\end{equation}
and
\begin{equation}
  E_4(U):=\dfrac{\mu m(1-m)}{U_\varepsilon^2}\left(\dfrac{(2-m)U_\xi^2}{U^{3-m}}-\dfrac{U_{\xi\xi}}{U^{2-m}}\right).
\end{equation}
From Lemma \ref{uxi1}, we have $|U_\xi|\leq CU^{2-m}$ for $\xi\in {\rm R}$, then we get
\begin{equation}\label{E1}
  |E_1(U)|\leq C\left(\dfrac{1}{U_\varepsilon^m}+\dfrac{1}{U_\varepsilon^{1+m}}\right)\leq \df{C}{U^{1+m}},~\mbox{for}~U\in[u_+,u_-].
\end{equation}
Combining with $|U_{\xi\xi}|\leq CU^{3-2m}$ for all $\xi\in {\rm R}$ and the Cauchy-Schwarz inequality for any $\sigma>0$,   for $U\in[u_+,u_-]$ it similarly holds that
\begin{equation}
  |E_2(U)||\phi_\xi||\phi_{\xi\xi}|\leq C\df{|\phi_\xi||\phi_{\xi\xi}|}{U_\varepsilon^{2m}}\leq C\df{|\phi_\xi||\phi_{\xi\xi}|}{U^{2}} \leq C\left(\df{\phi_\xi^2}{U^{1+m}}+\df{\phi_{\xi\xi}^2}{U^{3-m}}\right),
\end{equation}
\begin{equation}
  |E_3(U)||\phi_{\xi\xi\xi}||\phi_{\xi\xi}|\leq C\df{|\phi_{\xi\xi\xi}||\phi_{\xi\xi}|}{U_\varepsilon^{2}}\leq C\df{|\phi_{\xi\xi\xi}||\phi_{\xi\xi}|}{U_\varepsilon^{3-m}}\leq C\left(\df{1}{\sigma}\df{\phi_{\xi\xi}^2}{U^{3-m}}+{\sigma}\df{\phi_{\xi\xi\xi}^2}{U^{1-m}U_\varepsilon^{2}}\right),
\end{equation}
and
\begin{equation}\label{E4}
  |E_4(U)||\phi_{\xi\xi\xi}||\phi_{\xi}|\leq C\df{|\phi_{\xi\xi\xi}||\phi_{\xi}|}{U_\varepsilon^{1+m}}\leq C\df{|\phi_{\xi\xi\xi}||\phi_{\xi}|}{U_\varepsilon^{2}}\leq C\left(\df{1}{\sigma}\df{\phi_{\xi}^2}{U^{1+m}}+{\sigma}\df{\phi_{\xi\xi\xi}^2}{U^{1-m}U_\varepsilon^{2}}\right).
\end{equation}

By virtue of Lemma \ref{uxi1} and Lemma \ref{F}, we have
\begin{equation}
  \begin{array}{llll}
&\D\int_0^t\int_R
\left(-\dfrac{\phi_{\xi\xi}}{U_\varepsilon^2}\right)_\xi F_\xi
\;{\rm d}\xi{\rm d}\tau\\[14pt]
&\D\leq C\int_0^t\int_R \left(\dfrac{|\phi_{\xi\xi\xi}|}{U^2_\varepsilon}+\dfrac{|\phi_{\xi\xi}|}{U_\varepsilon^{1+m}}\right)\left(U^{2-m}|\phi_\xi|^2+|\phi_\xi||\phi_{\xi\xi}|\right) \;{\rm d}\xi{\rm d}\tau\\[14pt]
&\D\leq C  \int_0^t \left\|\df{\phi_\xi}{U}\right\|_{L^\infty} \int_R
{U^{1-m}|\phi_{\xi\xi\xi}||\phi_\xi|}+{U^{2-2m}|\phi_{\xi\xi}||\phi_\xi|}
+\dfrac{|\phi_{\xi\xi\xi}||\phi_{\xi\xi}|}{U_\varepsilon}+\dfrac{\phi_{\xi\xi}^2}{U^{m}_\varepsilon}
\;{\rm d}\xi{\rm d}\tau\\[14pt]
&\D\leq C\delta\int_0^t\int_R\df{\phi_{\xi\xi\xi}^2}{U^{1-m}U_\varepsilon^{2}}+\left(\df{1}{\sigma}+1\right)\df{\phi_{\xi\xi}^2}{U^{3-m}}+\left(\df{1}{\sigma}+1\right)\df{\phi_{\xi}^2}{U^{1+m}}\,{\rm d}\xi{\rm d}\tau.
    \end{array}
\end{equation}
By virtue of Lemma \ref{uxi1} and Lemma \ref{G}, we have
\begin{equation}\label{G91}
  \begin{array}{lllllll}
    &\D\int_0^t\int_R
\left(-\dfrac{\phi_{\xi\xi}}{U_\varepsilon^2}\right)_{\xi}G_{\xi\xi}
\;{\rm d}\xi{\rm d}\tau\\[14pt]
&\D\leq C\int_0^t\left\|\df{\phi_\xi}{U}\right\|_{L^\infty}\int_R
\dfrac{\phi_{\xi\xi\xi}^2}{ U^{1-m}U^2_\varepsilon}+\dfrac{|\phi_{\xi\xi\xi}||\phi_{\xi\xi}|}{U^2_\varepsilon}
+\dfrac{|\phi_{\xi\xi\xi}||\phi_{\xi}|}{U^{1+m}_\varepsilon}+\dfrac{|\phi_{\xi\xi}||\phi_{\xi}|}{U^{2m}_\varepsilon}+\dfrac{\phi_{\xi\xi}^2}{U^{1+m}_\varepsilon}
\;{\rm d}\xi{\rm d}\tau\\[14pt]
&\D\quad\;+C\int_0^t\int_R \left(\dfrac{|\phi_{\xi\xi\xi}|}{U^2_\varepsilon}+\dfrac{|\phi_{\xi\xi}||U_{\xi}|}{U_\varepsilon^{3}}\right)\dfrac{\phi_{\xi\xi}^2}{U^{2-m}}\;{\rm d}\xi{\rm d}\tau\\[14pt]
&\D\leq C\delta\int_0^t\int_R(\sigma+1)\df{\phi_{\xi\xi\xi}^2}{U^{1-m}U_\varepsilon^{2}}+\left(\df{1}{\sigma}+1\right)\df{\phi_{\xi\xi}^2}{U^{3-m}}+\df{1}{\sigma}\df{\phi_{\xi}^2}{U^{1+m}}\,{\rm d}\xi{\rm d}\tau\\[14pt]
&\D\quad\;+C\int_0^t\int_R \left(\dfrac{|\phi_{\xi\xi\xi}|}{U^2_\varepsilon}+\dfrac{|\phi_{\xi\xi}||U_{\xi}|}{U_\varepsilon^{3}}\right)\dfrac{\phi_{\xi\xi}^2}{U^{2-m}}\;{\rm d}\xi{\rm d}\tau.
  \end{array}
\end{equation}
As for the last term of \eqref{G91}, by using H\"{o}lder's inequality and Sobolev inequality and noting $\left\| \frac{\phi_{\xi\xi}(t)}{U}\right\|^2\leq CN_1(t)\leq C\delta$ , we have
\begin{equation}\label{g92}
\begin{array}{llll}
 &\D\int_0^t\int_R \dfrac{|\phi_{\xi\xi\xi}|}{U^2_\varepsilon }\df{\phi_{\xi\xi}^2}{U^{2-m}}\;{\rm d}\xi{\rm d}\tau\\[14pt]
 &\D\leq C\int_0^t \left\|\df{\phi_{\xi\xi}}{U_\varepsilon U^{\frac{1-m}{2}}}\right\|_{L^\infty}\int_R\df{|\phi_{\xi\xi\xi}||\phi_{\xi\xi}|}{U_\varepsilon U^\frac{3-m}{2}} \;{\rm d}\xi{\rm d}\tau\\[14pt]
 &\D\leq C\int_0^t \left\|\df{\phi_{\xi\xi}}{U_\varepsilon U^{\frac{1-m}{2}}}\right\|_{L^\infty} \left(\int_R\df{|\phi_{\xi\xi\xi}|^2}{U^{1-m}U_\varepsilon^2} \;{\rm d}\xi \right)^\frac{1}{2}
 \left(\int_R\df{\phi_{\xi\xi}^2}{U^2} \;{\rm d}\xi \right)^\frac{1}{2} {\rm d}\tau \\[14pt]
 &\D\leq C\delta \int_0^t \left\|\df{\phi_{\xi\xi}}{U_\varepsilon U^{\frac{1-m}{2}}}\right\|_{L^\infty} \left(\int_R\df{|\phi^2_{\xi\xi\xi}|}{U^{1-m}U_\varepsilon^2} \;{\rm d}\xi \right)^\frac{1}{2} {\rm d}\tau\\[14pt]
  &\D\leq C\delta \int_0^t  \left\|\df{\phi_{\xi\xi}}{U_\varepsilon U^{\frac{1-m}{2}}}\right\|^\frac{1}{2}\left\|\df{\phi_{\xi\xi\xi}}{U^{\frac{1-m}{2}}U_\varepsilon}\right\|^\frac{3}{2} {\rm d}\tau\\[14pt]
 &\D\leq C\delta \left(\int_0^t\int_R\df{\phi^2_{\xi\xi\xi}}{U^{1-m}U_\varepsilon^2} \;{\rm d}\xi {\rm d}\tau+\int_0^t\int_R \df{\phi_{\xi\xi}^2}{U^{3-m}}  \;{\rm d}\xi{\rm d}\tau\right).
  \end{array}
    \end{equation}
    The term $\int_0^t\int_R \frac{|\phi_{\xi\xi}||U_{\xi}|}{U_\varepsilon^{3}}\frac{\phi_{\xi\xi}^2}{U^{2-m}}\;{\rm d}\xi{\rm d}\tau$ can be similarly estimated.
 Substituting \eqref{E1}-\eqref{g92} into \eqref{32}, taking $\delta$ and $\sigma$ appropriately small and using Fatou's lemma, Lemma \ref{l1} and Lemma \ref{h1}, we obtain \eqref{H2}.

 By virtue of $\frac{1}{U^{2}}\sim {\langle \xi\rangle_+^{\alpha_1}}$ and $\frac{1}{U^{3-m}}\sim {\langle \xi\rangle_+^{\alpha_3}}$  for $\xi\in {\rm R}$, we get \eqref{h21} and  complete the proof of Lemma \ref{h2}.
\end{proof}

From Lemmas \ref{l1}-\ref{h2}, when we choose $\|\phi_\xi(t)\|_{1,\langle \xi\rangle_+^{\alpha_1}}^2+\|\phi(t)\|_{\langle \xi\rangle_+^{\alpha_2}}^2\leq \delta_0 \leq\frac{\delta}{C} $, we obtain \eqref{final}.  Now we prove Theorem \ref{ph1} based on Proposition \ref{prop}.

\begin{proof}[Proof of Theorem \ref{ph1}]

  Multiplying \eqref{Main} by $-\phi_{\xi\xi}$, we have
  \begin{equation}
   \df{1}{2}(\phi_\xi)^2_t=g'(U)\phi_\xi\phi_{\xi\xi}- \left(\mu m\dfrac{\phi_\xi}{U^{1-m}}\right)_\xi\phi_{\xi\xi}-F\phi_{\xi\xi}-\phi_{\xi\xi}G_\xi.
  \end{equation}
  Integrating the above equation with respect to $\xi $ over $(-\infty,+\infty)$ yields
  \begin{equation}
    \left|\df{{\rm d}}{{\rm d}t}(\|\phi_\xi\|^2)\right|\leq C\left(\|\phi_\xi\|^2+\left\|\df{\phi_{\xi\xi}}{U^{\frac{1-m}{2}}}\right\|^2\right)\leq C\left(\left\|\df{\phi_\xi}{U^{\frac{1+m}{2}}}\right\|^2+\left\|\df{\phi_{\xi\xi}}{U^{\frac{3-m}{2}}}\right\|^2\right).
  \end{equation}
  Considering \eqref{L1}, \eqref{H11} and \eqref{H2}, we get
  \begin{equation}
    \int_0^{+\infty}\left|\df{{\rm d}}{{\rm d}t}(\|\phi_\xi\|^2)\right| {\rm d} t\leq C \int_0^{+\infty}\left(\left\|\df{\phi_\xi}{U^{\frac{1+m}{2}}}\right\|^2+\left\|\df{\phi_{\xi\xi}}{U^{\frac{3-m}{2}}}\right\|^2\right){\rm d} t\leq C,
 \end{equation}
which implies that
\begin{equation}
  \|\phi_\xi(t)\|\rightarrow 0,~as~t\rightarrow +\infty.
\end{equation}
On the other hand, $\|\phi_{\xi\xi}\|$ is uniformly bounded in $t\geq 0$ due to \eqref{final}, then using Sobolev inequality, we obtain
\begin{equation}
   \phi^2_\xi(\xi,t)=2\int^\xi_{-\infty}\phi_x\phi_{xx}(x,t) {\rm d} x\leq 2\|\phi_{\xi}(t)\|\|\phi_{\xi\xi}(t)\|\rightarrow0,~as~t\rightarrow +\infty.
 \end{equation}
Thus, the proof is complete.
\end{proof}

\section{${A~priori}$ estimates for the case of  $f'(u_+)=s<f'(u_-)$}

In what follows, we confine ourselves to the proof of Proposition \ref{prop1} by establishing the $a~priori$~estimates for $f'(u_+)=s<f'(u_-)$.
Let $\phi(\xi,t)\in Y(0,T) $ be the solution of \eqref{Main}-\eqref{G1} with
\begin{equation}\label{pri2}
  N_2(t)\leq \delta,~\mbox{for}~t\in[0,T],
\end{equation}
which implies that
\begin{equation}\label{inf2}
  \sup_{t\in[0,T]}\left(\left\|{\phi(t)}\right\|_{L^\infty} +\left\|\df{\phi_\xi(t)}{U}\right\|_{L^\infty}\right)\leq C\delta.
\end{equation}

\begin{lem}\label{uxi2}
  In the case of $f'(u_+)=s<f'(u_-)$, it holds that
  \begin{equation}
    |U_\xi|\leq CU^{k_++2-m},
  \end{equation}
  and
  \begin{equation}
    |U_{\xi\xi}|\leq CU^{2k_++3-2m},
  \end{equation}
  for all $\xi\in{\rm R}$.
\end{lem}
\begin{proof}
  This lemma is a direct corollary of Theorem \ref{pro} for $s=f'(u_+)$.
\end{proof}

\begin{lem}\label{F2}
Under the $a~priori$ assumption \eqref{pri2}, it holds that
\begin{equation}
|F|\leq C\phi_\xi^2,
\end{equation}
and
\begin{equation}
|F_\xi|\leq C(U^{k_++2-m}|\phi_\xi|^2+|\phi_\xi||\phi_{\xi\xi}|).
 \end{equation}

\end{lem}

\begin{proof}
The proof is similar to the proof of Lemma \ref{F} in  the previous section, with the representation of $F=-\{f(U+\phi_\xi)-f(U)-f'(U)\phi_\xi\}$, Taylor's expansion and Lemma \ref{uxi2}. We skip the details.
\end{proof}

\begin{lem}\label{G2}
Under the $a~priori$ assumption \eqref{pri2}, it holds that
\begin{equation}
|G|\leq \dfrac{C\phi_\xi^2}{U^{2-m}},
\end{equation}
\begin{equation}
|G_\xi|\leq
\left\{U^{k_+-1}{|\phi_\xi|^2}+\dfrac{|\phi_\xi||\phi_{\xi\xi}|}{U^{2-m}}\right\},
\end{equation}
and
\begin{equation}
\begin{split}
|G_{\xi\xi}|\leq C
\left\{\dfrac{|\phi_{\xi\xi\xi}||\phi_\xi|}{U^{2-m}}+U^{2k_+-m}|\phi_\xi|^2+U^{k_+-1}{|\phi_{\xi\xi}||\phi_\xi|}+\dfrac{|\phi_{\xi\xi}|^2}{U^{2-m}}\right\}.
\end{split}
 \end{equation}
\end{lem}

\begin{proof}
From the representation of $G=\mu\{(U+\phi_\xi)^m-U^m- m\dfrac{\phi_\xi}{U^{1-m}}\}$, we similarly have
\begin{equation}
 \left| G_\xi\right|\leq C \left\{\df{|U_\xi||\phi_{\xi}|^2}{U^{3-m}}+\df{|\phi_{\xi\xi}||\phi_\xi|}{U^{2-m}}\right\},
\end{equation}
and
 \begin{equation}
\left|G_{\xi\xi}\right|
\leq C\left\{ \dfrac{|U_{\xi\xi}||\phi_\xi|^2}{U^{3-m}}+\dfrac{|\phi_{\xi\xi\xi}||\phi_\xi|}{U^{2-m}}
+\dfrac{|U_\xi|^2|\phi_\xi|^2}{U^{4-m}}+\dfrac{|U_\xi||\phi_{\xi\xi}||\phi_\xi|}{U^{3-m}}+\dfrac{|\phi_{\xi\xi}|^2}{U^{2-m}}\right\}.
 \end{equation}
By virtue of \eqref{inf2} and Theorem \ref{uxi2}, the proof of Lemma \ref{G2} is completed.
\end{proof}

\begin{lem}\label{le2}
Under the same conditions as in Proposition \ref{prop1}, it holds that
\begin{equation}\label{Le2}
\|\phi(t)\|_{w}^2+\int_0^t\int_R\dfrac{w\phi_\xi^2}{U^{1-m}}\,{\rm d}\xi{\rm d}\tau\leq C\|\phi(0)\|_{w}^2,
\end{equation}
for $t\in[0,T]$ provided $\delta\ll 1$. Thus it holds that
\begin{equation}\label{Lde}
\|\phi(t)\|_{\langle \xi\rangle_+^{\beta_2}}^2+\int_0^t\int_0^t\|\phi_\xi(\tau)\|_{\langle \xi\rangle_+}^2\,{\rm d}\tau\leq C\|\phi(0)\|_{\langle \xi\rangle_+^{\beta_2}}^2,
\end{equation}
with $\beta_2={\frac{k_+}{k_++1-m}}$ for $t\in[0,T]$ provided $\delta\ll 1$.
\end{lem}

\begin{proof}
 We multiply \eqref{Main} by $w(U)\phi$ to obtain
 \begin{equation}\label{5}
   \begin{array}{lll}
  &\D\left(\df{w(U)}{2}\phi^2\right)_t+\left((wg)'(U)\df{\phi^2}{2}-\mu m w(U)\df{\phi\phi_\xi}{U^{1-m}}-w(U)\phi G\right)_\xi\\[14pt]
  &\D-(wg)''(U)U_\xi\df{\phi^2}{2}+\left(\df{\mu m w'(U)U_\xi}{U^{1-m}}-(w'g)(U)\right)\phi\phi_\xi+\mu m w(U)\df{\phi_\xi^2}{U^{1-m}}\\[14pt]
  &\D=Fw(U)\phi-w'(U)U_\xi G \phi-w(U) G\phi_\xi.
\end{array}
 \end{equation}
 Note that $\frac{\mu m w'(U)U_\xi}{U^{1-m}}-(w'g)(U)=0$ referring to \eqref{Hu}. Furthermore, we have  $(wg)''(U)=2$, which leads to
\begin{equation}
  -(wg)''(U)U_\xi\df{\phi^2}{2}= |U_\xi|\phi^2,~U\in[u_+,u_-].
\end{equation}
Multiplying \eqref{5} by the cut-off function $\eta(\xi)$ defined in Lemma \ref{cutoff} and integrating the resultant equation over ${\rm R}$, we get
\begin{equation}\label{de}
\begin{array}{lll}
  &\df{{\rm d}}{{\rm d}t}\displaystyle \int_{-2L}^{2L}\D\df{w(U)}{2}\phi^2\eta\,{\rm d}\xi+\int_{-2L}^{2L}|U_\xi|\phi^2\eta\,{\rm d}\xi+\int_{-2L}^{2L}\mu m\dfrac{w(U)\phi_\xi^2}{U^{1-m}}\eta\,{\rm d}\xi{\rm d}\tau\\[14pt]
  & \D =\int_{-2L}^{2L}Fw(U)\phi\eta\,{\rm d}\xi\,{\rm d}\xi+\int_{-2L}^{2L}-w'(U)U_\xi G \phi\eta\,{\rm d}\xi+\int_{-2L}^{2L}-Gw(U) \phi_\xi\eta\,{\rm d}\xi\\[14pt]
&\D+\int_{-2L}^{2L}\eta_\xi(\xi)\left((wg)'(U)\df{\phi^2}{2}-\mu m w(U)\df{\phi\phi_\xi}{U^{1-m}}-w(U)\phi G\right)\,{\rm d}\xi.
\end{array}
\end{equation}
Note that by virtue of \eqref{Nt2} and \eqref{pri2}, we have ${\phi}\in L^2_{loc}({\rm R})$ and
  $\frac{\phi_\xi}{U}\in L^2_{loc}({\rm R})$
  for $t\in[0,T]$. Furthermore, it holds that
   \begin{equation}
    \int_{-2L}^{2L}{\phi^2}+\df{\phi^2}{U^2}\,{\rm d}\xi\leq  \int_R{\phi^2}+\df{\phi^2}{U^2}\,{\rm d}\xi\leq N_2(t)\leq \delta.
  \end{equation}
Then, we get
\begin{equation}
  \int_{-2L}^{2L}\eta_\xi(\xi)(wg)'(U)\df{\phi^2}{2}\,{\rm d}\xi\leq C\int_{-2L}^{2L}|\eta_\xi(\xi)|{\phi^2}\,{\rm d}\xi\leq \df{C}{L}\int_{-2L}^{2L}{\phi^2}\,{\rm d}\xi\leq \df{C\delta}{L}.
\end{equation}
Thanks to the Cauchy-Schwarz inequality, we have
\begin{equation}
\begin{array}{lllll}
  &\D\int_{-2L}^{2L}-\mu m\eta_\xi(\xi) w(U)\df{\phi\phi_\xi}{U^{1-m}}\,{\rm d}\xi\\[14pt]
  &\D\leq C \int_{-2L}^{2L}|\eta_\xi(\xi)| |w(U)|\df{|\phi||\phi_\xi|}{U^{1-m}}\,{\rm d}\xi\\[14pt]
  &\D\leq \df{C}{L} \int_{-2L}^{2L}\phi^2+\df{\phi_\xi^2}{U^{2-2m}}\,{\rm d}\xi\\[14pt]
  & \D\leq \df{C}{L}\int_{-2L}^{2L}\phi^2+\df{\phi_\xi^2}{U^{2}}\,{\rm d}\xi\\[14pt]
  &\D\leq \df{C\delta}{L}.
\end{array}
 \end{equation}
 Similarly, by virtue of Lemma \ref{cutoff}, \eqref{inf2} and Lemma \ref{G2}, we have
 \begin{equation}
 \begin{array}{llll}
    &\D\int_{-2L}^{2L}-w(U)\phi G\eta_\xi(\xi)\,{\rm d}\xi\\[14pt]
    &\D\leq C\int_{-2L}^{2L}\df{|w(U)||\phi||\phi_\xi|^2 |\eta_\xi(\xi)|}{U^{2-m}}\,{\rm d}\xi\\[14pt]
    &\D\leq \df{C}{L}\left\|{\phi(t)}\right\|_{L^\infty} \int_{-2L}^{2L}\df{\phi_\xi^2}{U^{2-m}}\,{\rm d}\xi\\[14pt]
    &\D\leq \df{C\delta}{L}\int_{-2L}^{2L}\df{\phi_\xi^2}{U^{2}}\,{\rm d}\xi\\[14pt]
    &\D\leq\df{C\delta^2}{L}.
 \end{array}
 \end{equation}
 Hence, we obtain
 \begin{equation}\label{de23}
\begin{array}{lll}
  &\df{{\rm d}}{{\rm d}t}\displaystyle \int_{-2L}^{2L}\D\df{w(U)}{2}\phi^2\eta\,{\rm d}\xi+\int_{-2L}^{2L}|U_\xi|\phi^2\eta\,{\rm d}\xi+\int_{-2L}^{2L}\mu m \dfrac{w(U)\phi_\xi^2}{U^{1-m}}\eta\,{\rm d}\xi\\[14pt]
  & \D \leq C\left(\int_{-2L}^{2L}Fw(U)\phi\eta\,{\rm d}\xi+\int_{-2L}^{2L} -w'(U)U_\xi G \phi\eta\,{\rm d}\xi+\int_{-2L}^{2L}-Gw(U) \phi_\xi\eta\,{\rm d}\xi+\df{\delta+\delta^2}{L}\right).
\end{array}
\end{equation}
Integrating the above equation over $[0,t]$ and taking $L\rightarrow+\infty$, we have
\begin{equation}\label{de33}
\begin{array}{lll}
  &\D\|\phi(t)\|_{w}^2+\int_0^t\int_R|U_\xi|\phi^2\,{\rm d}\xi{\rm d}\tau+\int_0^t\int_R\dfrac{w(U)\phi_\xi^2}{U^{1-m}}\,{\rm d}\xi{\rm d}\tau\\[14pt]
  & \D \leq C\|\phi(0)\|_{w}^2+C\int_0^t\int_R |Fw(U)\phi|\,{\rm d}\xi{\rm d}\tau+C\int_0^t\int_R|w'(U)U_\xi G \phi|\,{\rm d}\xi{\rm d}\tau\\[14pt]
  &\D\quad+C\int_0^t\int_R|Gw(U) \phi_\xi|\,{\rm d}\xi{\rm d}\tau.
\end{array}
\end{equation}
For the second term on the right hand side of \eqref{de33}, by applying \eqref{inf2} and Lemma \ref{F2}, we get
\begin{equation}\label{l18}
\begin{array}{lll}
&\D \int_0^t\int_R|Fw(U)\phi|\,{\rm d}\xi{\rm d}\tau \\[14pt]
&\D\leq C\int_0^t\int_R w(U)\phi_\xi^2|\phi|\,{\rm d}\xi{\rm d}\tau\\[14pt]
 &\D \leq C\int_0^t\left\|{\phi}\right\|_{L^\infty}\int_R{w(U)\phi_\xi^2}\,{\rm d}\xi{\rm d}\tau\\[14pt]
 &\D\leq C\delta \int_0^t\int_R\df{w(U)\phi_\xi^2}{U^{1-m}}\,{\rm d}\xi{\rm d}\tau.
\end{array}
\end{equation}
Similarly, by virtue of \eqref{inf2} and Lemma \ref{G2}, for the last term on the right hand side of \eqref{de33}, we have
\begin{equation}\label{l19}
\begin{array}{lll}
   &\D \int_0^t\int_R|Gw(U)\phi_\xi|\;{\rm d}\xi{\rm d}\tau \\[14pt]
   &\D\leq C\int_0^t\int_R \df{w(U)|\phi_\xi|^3}{U^{2-m}}{\rm d}\xi{\rm d}\tau\\[14pt]
  &\D\leq C\int_0^t\left\|\df{\phi_\xi}{U}\right\|_{L^\infty}\int_R\df{w(U)|\phi_\xi|^2}{U^{1-m}}\,{\rm d}\xi{\rm d}\tau\\[14pt]
  &\D\leq C\delta\int_0^t\int_R\df{w(U)|\phi_\xi|^2}{U^{1-m}}\,{\rm d}\xi{\rm d}\tau.
\end{array}
\end{equation}
Next, we estimate  the third term on the right hand side of \eqref{de}. When  $s=f'(u_+)$, we have
\begin{equation}
  |g(U)|\sim {|U|^{k_++1}},~|U_\xi|\sim |U|^{k_++2-m},~\mbox{as}~U\rightarrow u_+,
  \end{equation}
  and
  \begin{equation}
    |g(U)|\sim {|U-u_-|},~|U_\xi|\sim |U-u_-|,~\mbox{as}~U\rightarrow u_-,
  \end{equation}
which implies that
\begin{equation}\label{wg}
  \left|\df{w'(U)}{w(U)}\right|\left|\df{U_\xi}{U}\right|=\left|\df{(2U-u_-)g(U)-g'(U)U(U-u_-)}{g(U)U(U-u_-)}\right|\left|\df{U_\xi}{U}\right|\sim O(1),~\mbox{as}~U\rightarrow u_\pm.
  \end{equation}
  Since there is no singularity for $U\in(u_+,u_-)$, we conclude that \eqref{wg} holds for all $U\in[u_+,u_-]$. We get
  \begin{equation}\label{l21}
  \begin{array}{lll}
    &\D\int_0^t\int_R|w'(U)U_\xi G\phi|\,{\rm d}\xi{\rm d}\tau\\[14pt]
    &\D\leq C\int_0^t\int_R   \left|\df{w'(U)}{w(U)}\right|\left|\df{U_\xi}{U}\right| \df{w(U)\phi_\xi^2|\phi|}{U^{1-m}}\,{\rm d}\xi{\rm d}\tau\\[14pt]
    &\D\leq C\int_0^t\left\|{\phi}\right\|_{L^\infty}\int_R\df{w(U)\phi_\xi^2}{U^{1-m}}\,{\rm d}\xi{\rm d}\tau\\[14pt]
  &\D\leq C\delta\int_0^t\int_R\df{w(U)\phi_\xi^2}{U^{1-m}}\,{\rm d}\xi{\rm d}\tau.
  \end{array}
\end{equation}
Substituting \eqref{l18}, \eqref{l19} and \eqref{l21} into \eqref{de},
taking $\delta$ sufficiently small and determining an appropriate constant parameter $C$, we get \eqref{Le2}.

When $f'(u_+)=s<f'(u_-)$, we have
\begin{equation}
  |w(U)|\sim |U|^{-k_+},~U\sim |\xi|^{-\frac{1}{k_++1-m}},~\mbox{as}~\xi\rightarrow +\infty,
\end{equation}
and
\begin{equation}
  |w(U)|\sim O(1),~\mbox{as}~\xi\rightarrow -\infty.
\end{equation}
We get  that $w(U)\sim {\langle \xi\rangle_+^{\beta_2}}$ and $\frac{w(U)}{U^{1-m}}\sim  {\langle \xi\rangle_+}$ for $\xi\in {\rm R}$, which implies \eqref{Lde}.

Thus the proof is completed.
\end{proof}

Next, we estimate $\phi_\xi$. Using the same weight function as in Lemma \ref{h1}, Lemma \ref{eh1} is derived in a similar manner as follows. It's worth noting that the difference in the proof lies in the different decay rates of the viscous shock waves at far filed as $\xi\rightarrow+\infty$. In this degenerate case, we have $|U_\xi|\leq CU^{k_++2-m}$ and $|U_{\xi\xi}|\leq CU^{2k_++3-2m}$ for $U\in[u_+,u_-]$ .

\begin{lem}\label{eh1}
Under the same conditions as in Proposition \ref{prop1}, it holds that
\begin{equation}\label{eH11}
\left\|\df{{\phi_\xi(t)}}{U}\right\|^2+\int_0^t\int_R \dfrac{\phi_{\xi\xi}^2}{U^{3-m}} \,{\rm d}\xi{\rm d}\tau \leq C\left\|\df{{\phi_\xi(0)}}{U}\right\|^2
+ C\left\|{{\phi(0)}}\right\|_w^2,
\end{equation}
for $t\in[0,T]$ provided $\delta\ll 1$. Thus, it holds that
\begin{equation}\label{Hde}
\|\phi_\xi(t)\|_{\langle \xi\rangle_+^{\beta_1}}^2+\int_0^t\|\phi_\xi(\tau)\|_{\langle \xi\rangle_+^{\beta_3}}^2\,{\rm d}\tau\leq C\|\phi_\xi(0)\|_{\langle \xi\rangle_+^{\beta_1}}^2+C\|\phi(0)\|_{\langle \xi\rangle_+^{\beta_2}}^2,
\end{equation}
with  $\beta_1={\frac{2}{k_++1-m}}$, $\beta_2={\frac{k_+}{k_++1-m}}$ and $\beta_3={\frac{3-m}{k_++1-m}}$ for $t\in[0,T]$ provided $\delta\ll 1$.
\end{lem}

\begin{proof}

Multiplying \eqref{Main} by $-\left(\dfrac{\phi_\xi}{U^2_\varepsilon}\right)_\xi$, we get
\begin{equation}\label{eh1s}
\begin{split}
 & \left(\dfrac{1}{2}\dfrac{\phi^2_\xi}{U^2_\varepsilon}\right)_t -\left(\df{\phi_t\phi_\xi}{U^2_\varepsilon}+g'(U)\dfrac{\phi^2_\xi}{U^2_\varepsilon}-\df{\phi_\xi}{U_\varepsilon^2}F\right)_\xi+Q_1(U)\phi_\xi^2+\mu m \dfrac{1}{U^{1-m}U^2_\varepsilon}\phi^2_{\xi\xi}+Q_2(U)\phi_\xi\phi_{\xi\xi}\\
  &=\dfrac{\phi_\xi}{U^2_\varepsilon} F_\xi-\left(\dfrac{\phi_\xi}{U^2_\varepsilon}\right)_\xi G_{\xi},
\end{split}
\end{equation}
where
\begin{equation}
\begin{array}{ll}
  Q_1(U)=g''(U)\dfrac{U_\xi}{U_\varepsilon^2}+2\mu m(1-m)\dfrac{U^2_\xi}{U^{2-m}U_\varepsilon^3},
\end{array}
\end{equation}
and
\begin{equation}
  Q_2(U)=\dfrac{g'(U)}{U_\varepsilon^2}-\dfrac{\mu m(1-m)U_\xi}{U^{2-m}U^2_\varepsilon}
-\dfrac{2\mu mU_{\xi}}{U^{1-m}U_\varepsilon^3}.
\end{equation}
$Q_1(U)$ and $Q_2(U)$ here are defined in the same way  as in Lemma \ref{h1}.
By virtue of the Cauchy-Schwarz inequality for any $\sigma>0$ and $k_+\geq1$, we have
\begin{equation}
 | Q_1(U)|\leq C\left(\df{|U_\xi|}{U^2}+\df{|U_\xi|^2}{U^{5-m}}\right) \leq C\left({U^{k_+-m}}+{U^{2k_+-1-m}}\right)\leq{C},~\mbox{for}~U\in[u_+,u_-],
\end{equation}
and
\begin{equation}
  \left|Q_2(U)\phi_\xi\phi_{\xi\xi}\right|\leq C|U_\varepsilon^{k_+-2}\phi_\xi\phi_{\xi\xi}|\leq C\df{|\phi_\xi\phi_{\xi\xi}|}{U_\varepsilon}\leq C\left(\sigma \dfrac{\phi^2_{\xi\xi}}{U_\varepsilon^2}+\dfrac{1}{\sigma}{\phi^2_{\xi}}\right),~\mbox{for}~U\in[u_+,u_-].
\end{equation}
Integrating \eqref{eh1s} over ${\rm R}\times [0,t]$, we get
\begin{equation}\label{eh1s6}
\begin{array}{ll}
  &\D\int_R \dfrac{\phi^2_\xi(t)}{U^2_\varepsilon}\,{\rm d}\xi+\int_0^t\int_R \dfrac{\phi_{\xi\xi}^2}{U^{1-m}U_\varepsilon^2} \,{\rm d}\xi{\rm d}\tau\\[14pt]
  &\D\leq C\int_R \dfrac{\phi^2_\xi(0)}{U^2_\varepsilon}\,{\rm d}\xi +C\sigma \int_0^t\int_R\dfrac{\phi^2_{\xi\xi}}{U^{1-m}U_\varepsilon^2}\,{\rm d}\xi{\rm d}\tau\\[14pt]
&\D\quad+C\left(\dfrac{1}{\sigma}+1\right)\int_0^t\int_R{\phi^2_{\xi}}\,{\rm d}\xi{\rm d}\tau\D+C\int_0^t\int_R\dfrac{\phi_\xi}{U^2_\varepsilon} F_\xi\,{\rm d}\xi{\rm d}\tau-C\int_0^t\int_R\left(\dfrac{\phi_\xi}{U^2_\varepsilon}\right)_\xi G_{\xi}\,{\rm d}\xi{\rm d}\tau.
\end{array}
\end{equation}
 Thanks to Lemma \ref{le2}, we have
\begin{equation}\label{et39}
  \int_0^t\int_R{\phi^2_{\xi}}\,{\rm d}\xi{\rm d}\tau\leq C\int_0^t\int_R\df{w(U)\phi^2_{\xi}}{U^{1-m}}\,{\rm d}\xi{\rm d}\tau\leq C\left\|{{\phi(0)}}\right\|_w^2.
\end{equation}
As for the last two terms of \eqref{eh1s6}, by virtue of Lemma \ref{uxi2}-\ref{G2}, we have
\begin{equation}
   \begin{array}{lll}
    &\D \int_0^t\int_R\dfrac{\phi_\xi}{U^2_\varepsilon} F_\xi\,{\rm d}\xi{\rm d}\tau\\[14pt]
    &\D\leq C\int_0^t\int_R{U^{k_+-m}}|\phi_\xi|^3+\df{|\phi_\xi|^2|\phi_{\xi\xi}|}{U^2_\varepsilon}\,{\rm d}\xi{\rm d}\tau\\[14pt]
    &\D\leq C  \int_0^t\left\|\df{\phi_\xi}{U}\right\|_{L^\infty}\int_RU^{k_++1-m}{\phi_\xi^2}+\df{|\phi_\xi||\phi_{\xi\xi}|}{U_\varepsilon}\,{\rm d}\xi{\rm d}\tau\\[14pt]
    &\D\leq C\delta\int_0^t\int_R\left(\dfrac{1}{\sigma}+1\right){\phi_\xi^2}+\sigma\df{\phi_{\xi\xi}^2}{U^2_\varepsilon}\,{\rm d}\xi{\rm d}\tau\\[14pt]
    &\D\leq C\delta\int_0^t\int_R\left(\dfrac{1}{\sigma}+1\right){\phi_\xi^2}+\sigma\df{\phi_{\xi\xi}^2}{U^{1-m}U^2_\varepsilon}\,{\rm d}\xi{\rm d}\tau,
   \end{array}
 \end{equation}
 and
 \begin{equation}\label{et41}
   \begin{array}{llll}
       &\D\int_0^t\int_R-\left(\dfrac{\phi_\xi}{U_\varepsilon^2}\right)_\xi G_{\xi}\,{\rm d}\xi{\rm d}\tau\\[14pt]
&\D\leq C\int_0^t\int_R \left(\dfrac{|\phi_{\xi\xi}|}{U^2_\varepsilon}+\dfrac{|\phi_\xi|| U_{\xi}|}{U^{3}_\varepsilon}\right)\left({U}^{k_+-1}{|\phi_\xi|^2}+\dfrac{|\phi_\xi||\phi_{\xi\xi}|}{U^{2-m}}
\right)\,{\rm d}\xi{\rm d}\tau\\[14pt]
&\D\leq C\int_0^t\left\|\df{\phi_\xi}{U}\right\|_{L^\infty} \int_R \dfrac{ |\phi_\xi||\phi_{\xi\xi}|}{U_\varepsilon^{2-k_+}}
+{U^{2k_+-1-m}|\phi_\xi|^2}
+\dfrac{|\phi_{\xi\xi}|^2}{U^{1-m}U^2_\varepsilon}+\df{|\phi_\xi||\phi_{\xi\xi}|}{U_\varepsilon}\,{\rm d}\xi{\rm d}\tau\\[14pt]
&\D\leq C\delta\int_0^t\int_R |\phi_\xi|^2
+\dfrac{|\phi_{\xi\xi}|^2}{U^{1-m}U^2_\varepsilon}+\df{|\phi_\xi||\phi_{\xi\xi}|}{U_\varepsilon}\,{\rm d}\xi{\rm d}\tau\\[14pt]
&\D\leq C\delta\int_0^t\int_R\left(\dfrac{1}{\sigma}+1\right){\phi_\xi^2}+(\sigma+1)\df{\phi_{\xi\xi}^2}{U^{1-m}U^2_\varepsilon}\,{\rm d}\xi{\rm d}\tau.
   \end{array}
 \end{equation}
  Substituting \eqref{et39}-\eqref{et41} into \eqref{eh1s6}, taking $\delta$ and $\sigma$ appropriately small and using Fatou's lemma,  we get \eqref{eH11}.

When $f'(u_+)=s<f'(u_-)$, we have
\begin{equation}
U(\xi)\sim\left\{
  \begin{array}{ll}
    |\xi|^{-\frac{1}{k_++1-m}},&~\mbox{as}~\xi\rightarrow+\infty,\\
1,&~\mbox{as}~\xi\rightarrow-\infty,
  \end{array}
  \right.
  \end{equation}
which implies that $\frac{1}{U^{2}}\sim {\langle \xi\rangle_+^{\beta_1}}$ and  $\frac{1}{U^{3-m}}\sim {\langle \xi\rangle_+^{\beta_3}}$ for $\xi\in {\rm R}$. We derive \eqref{Hde}.
 \end{proof}

\begin{lem}\label{eh2}
Under the same conditions as in Proposition \ref{prop1}, it holds that
\begin{equation}\label{eH21}
\left\|\df{\phi_{\xi\xi}(t)}{U}\right\|^2+\int_0^t\int_R \dfrac{\phi_{\xi\xi\xi}^2}{U^{3-m}} \,{\rm d}\xi{\rm d}\tau \leq C\left\|{\df{\phi_{\xi}(0)}{U}}\right\|_{1}^2
+ C\left\|{\phi(0)}\right\|_w^2,
\end{equation}
for $t\in[0,T]$ provided $\delta\ll 1$,
which implies that for the case $s=f'(u_+)$,
\begin{equation}\label{eH22}
\|\phi_{\xi\xi}(t)\|_{\langle \xi\rangle_+^{\beta_1}}^2+\int_0^t\|\phi_{\xi\xi\xi}(\tau)\|_{\langle \xi\rangle_+^{\beta_3}}^2\,{\rm d}\tau\leq C\|\phi_\xi(0)\|_{1,\langle \xi\rangle_+^{\beta_1}}^2+C\|\phi(0)\|_{\langle \xi\rangle_+^{\beta_2}}^2,
\end{equation}
with  $\beta_1={\frac{2}{k_++1-m}}$, $\beta_2={\frac{k_+}{k_++1-m}}$ and $\beta_3={\frac{3-m}{k_++1-m}}$ for $t\in[0,T]$ provided $\delta\ll 1$.
\end{lem}
\begin{proof}
  Similarly to the proof of Lemma \ref{h2}, differentiating \eqref{Main} once with respect to $\xi$ and
multiplying it by $-\left(\frac{\phi_{\xi\xi}}{U^2_\varepsilon}\right)_\xi$,  we then integrate it over ${\rm R}\times [0,t]$ and make use of  \eqref{Nt2}, \eqref{inf2} and Lemmas \ref{uxi2}-\ref{eh1} to obtain Lemma \ref{eh2}. Further  details
are omitted here for brevity.
\end{proof}

From Lemmas \ref{le2}-\ref{eh2}, when we choose $\|\phi_\xi(t)\|_{1,\langle \xi\rangle_+^{\beta_1}}^2+\|\phi(t)\|_{\langle \xi\rangle_+^{\beta_2}}^2\leq \delta_0 \leq\frac{\delta}{C} $, we obtain \eqref{finalde}.
Based on Proposition \ref{prop1}, we can similarly prove Theorem \ref{ph2} as we did with Lemma \ref{ph1}.

\begin{proof}[Proof of Theorem \ref{ph2}]
We omit the details.
\end{proof}

\section{Examples and Numerical simulations}

In this section we provide some numerical examples. For simplicity, let $\mu=1$, $u_-=1$ and $u_+=0$.

{\bf Example 1.}~~
Let $f(u)=u^2$, for $0<m\leq\frac{1}{2}$. We consider the Cauchy problem for
\begin{equation}\label{e1}
  u_t+(u^2)_x=\mu (u^m)_{xx},\;\;x\in \rm{R},\,t>0.
\end{equation}
Based on Theorem \ref{pro} for the existence of viscous shock waves, the equation \eqref{e1} admits a unique (up to shift) viscous shock wave solution $U(x-st)$ with the shock speed $s=\frac{f(u_-)}{u_-}=1$. Obviously, $f(u)$ is convex and satisfies Lax's entropy condition $f'(0)<s<f'(1)$. The convexity of $ K(u)=\frac{g(u)}{u^{2m}}=\frac{u(u-1)}{u^{2m}}$ for the stability of viscous shock is also held, namely,
\begin{equation*}
  K''(u)=2(1-m)(1-2m)u^{-2m}+2m(1-2m)u^{-2m-1}\geq0,~\mbox{for}~0<m\leq \df{1}{2}~\mbox{and}~u\in[0,1].
\end{equation*}
In particular, according to \eqref{210} in Theorem \ref{pro}, the viscous shock $U(x-st)=U(\xi)$ behaves as
\begin{eqnarray*}
&&|U(\xi)-0|=O(1)|\xi|^{-\frac{1}{1-m}} \ \ \mbox{ as } \xi\to +\infty, \\
&& |U(\xi)-1|=O(1)e^{-|\xi|/m} \ \ \mbox{ as } \xi\to -\infty.
\end{eqnarray*}
In order to keep the initial data as a perturbation of such a viscous shock, from the stability Theorem \ref{Mthm}, we need the initial data $u_0(x)$ to possess the same decay as the viscous shock $U(x)$ at the far fields, thus we choose  the initial value as
\begin{equation*}
   u_0(x)=\left\{
   \begin{array}{ll}
   \frac{1}{2}{\left(\frac{ 1-m}{ m}x+1\right)^{-\frac{1}{1-m}}},~&x\geq 0,\\[6pt]
   1-\frac{1}{2}e^{\frac{x}{ m}},~&x<0,
   \end{array}\right.
 \end{equation*}
 where $0<m\leq \frac{1}{2}$.
 \noindent

 Now we carry out some numerical computations by selecting different indices $m$: $m=0.5$, \ 0.3, \ 0.1, \ 0.05, respectively. From these numerical results (see Figure 1), we see that the solutions behave like viscous shock waves as time increases. In particular, we observe that the shape of solutions $u(t,x)$ steepens when $m$ approaches $0^+$. This indicates that the effect of the singular fast diffusion to the solution is essential.
\begin{figure}[htbp]
	\begin{center}
		\includegraphics[width=7cm]{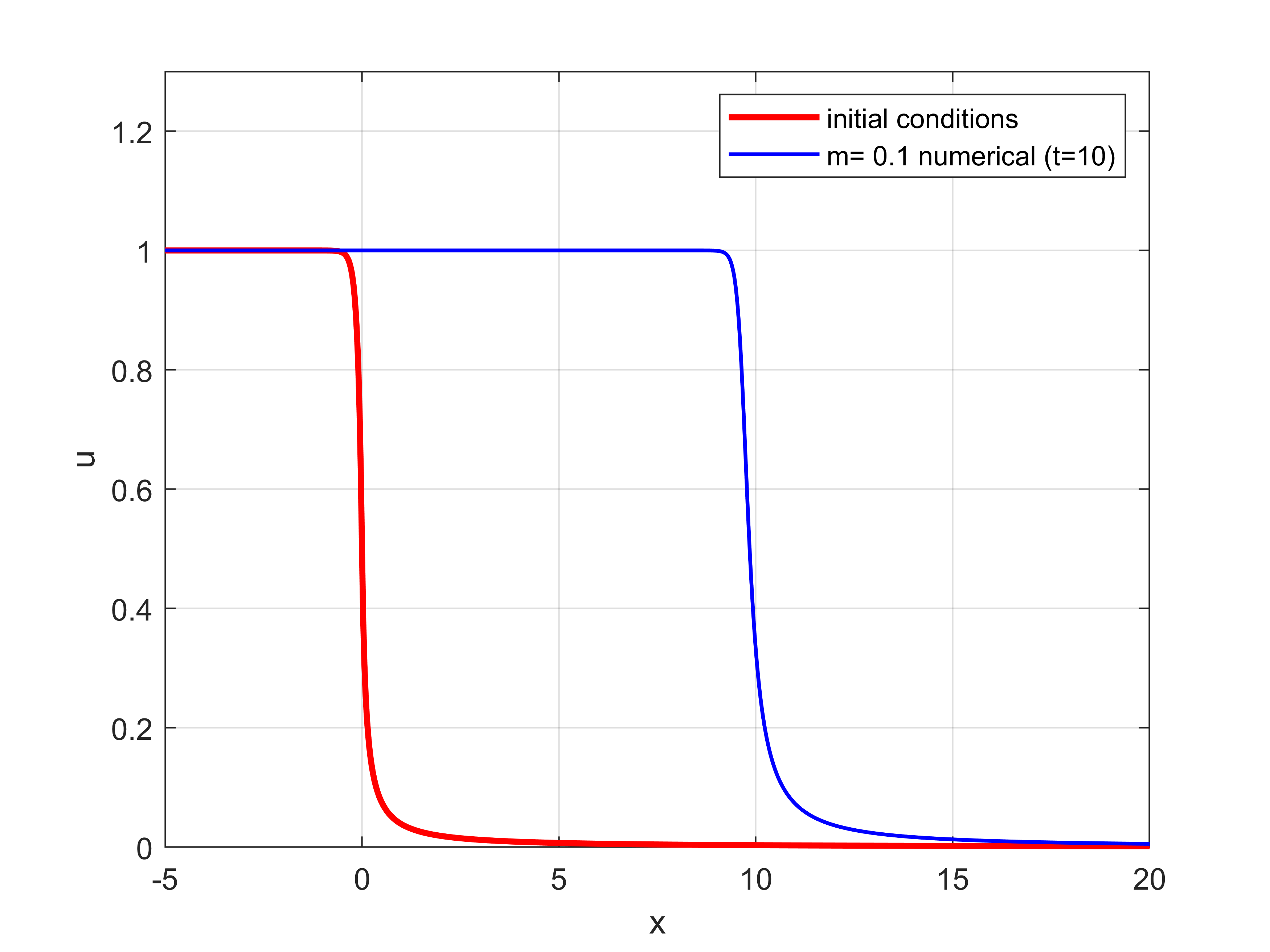}
		\includegraphics[width=7cm]{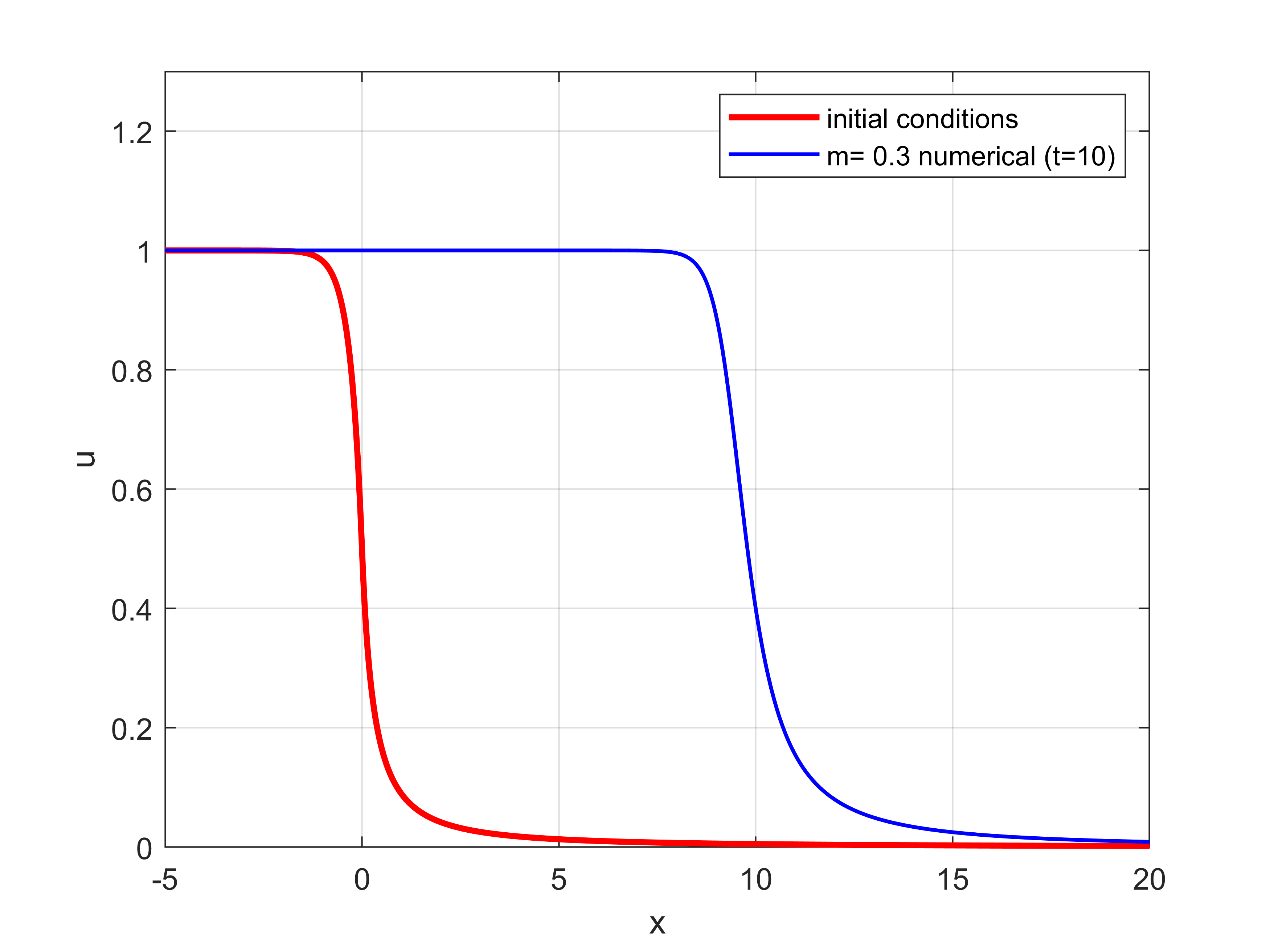}
        \includegraphics[width=7cm]{case1m=01.PNG}
		\includegraphics[width=7cm]{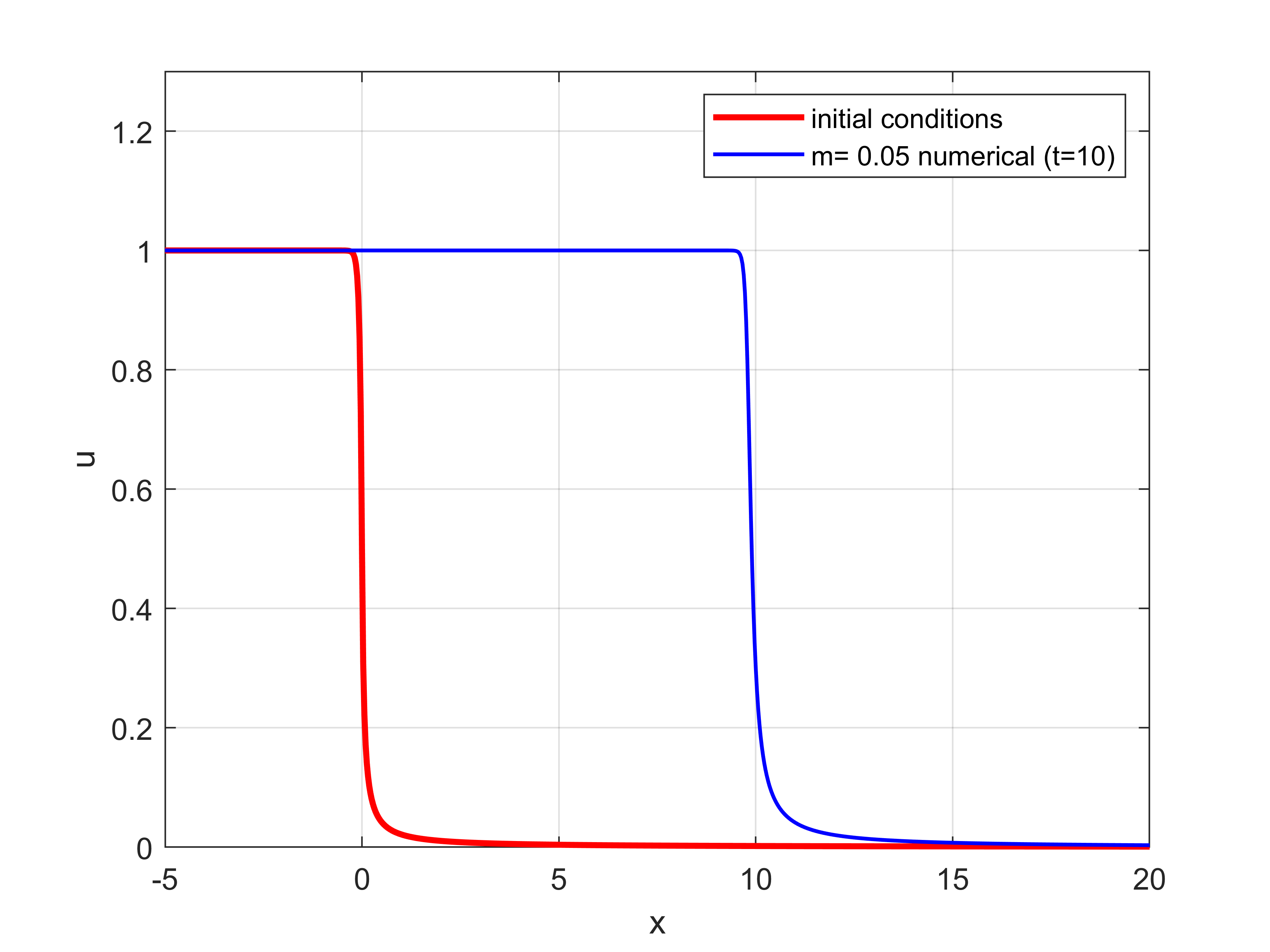}
		\caption{Case 1: $m=0.5, \ 0.3, \ 0.1, \ 0.05$. The solution $u(x,t)$ behaves like a monotone viscous shock wave, and the shapes of the solutions become increasingly steep as the fast-diffusion index $m$ decreases.}
		\label{fig1}
	\end{center}
\end{figure}

{\bf Example 2.}~~
Let $f(u)=2u^{3+2m}-u^{1+2m}$, for $0<m\leq \frac{1}{2}$. The targeted equation is
\begin{equation}\label{e2}
  u_t+(2u^{3+2m}-u^{1+2m})_x=\mu (u^m)_{xx},\;\;x\in \rm{R},\,t>0.
\end{equation}
In this case, the expected viscous shocks $U(x-st)$ are with the speed of $s=\frac{f(u_-)}{u_-}=1$ satisfying  $f'(0)<s<f'(1)$, and from \eqref{210} that
\begin{eqnarray*}
&&|U(\xi)-0|=O(1)|\xi|^{-\frac{1}{1-m}} \ \ \mbox{ as } \xi\to +\infty, \\
&& |U(\xi)-1|=O(1)e^{-\frac{4+2m}{m}|\xi|} \ \ \mbox{ as } \xi\to -\infty.
\end{eqnarray*}
With such information, we take the initial data as
\begin{equation*}
   u_0(x)=\left\{
   \begin{array}{ll}
   \frac{1}{2}{\left(\frac{ (4+2m)(1-m)}{ m}x+1\right)^{-\frac{1}{1-m}}},~&x\geq 0,\\[6pt]
   1-\frac{1}{2}e^{\frac{4+2m}{ m}x},~&x<0,
   \end{array}\right.
 \end{equation*}
 where $0<m\leq \frac{1}{2}$.

 For the above selected non-convex $f(u)$, there exists a unique $m_1^*:=\sqrt{\frac{m(2m+1)}{2(3+2m)(1+m)}}$ such that $f''(m_1^*)=0$ and
\begin{equation*}
  f''(u)\left\{
  \begin{array}{ll}
    <0,~u\in[0,m_1^*),\\[8pt]
    >0,~u\in(m_1^*,1].
  \end{array}\right.
\end{equation*}
In this case, we have $g(u)=f(u)-su=2u^{3+2m}-u^{1+2m}-u$ and $K(u)=\dfrac{g(u)}{u^{2m}}=2u^{3}-u-u^{1-2m}$ satisfying the sufficient condition for the stability
\begin{equation*}
  K''(u)=12u+2m(1-2m)u^{-2m-1}\geq0,~\mbox{for}~0< m\leq\df{1}{2}~\mbox{and}~u\in[0,1].
\end{equation*}
We carry out the numerical computations by selecting $m=0.2$. Figure 2 shows that the solution $u(t,x)$ behaves like a viscous shock as time increases.

\begin{figure}[htbp]
	\begin{center}
		\includegraphics[width=7cm]{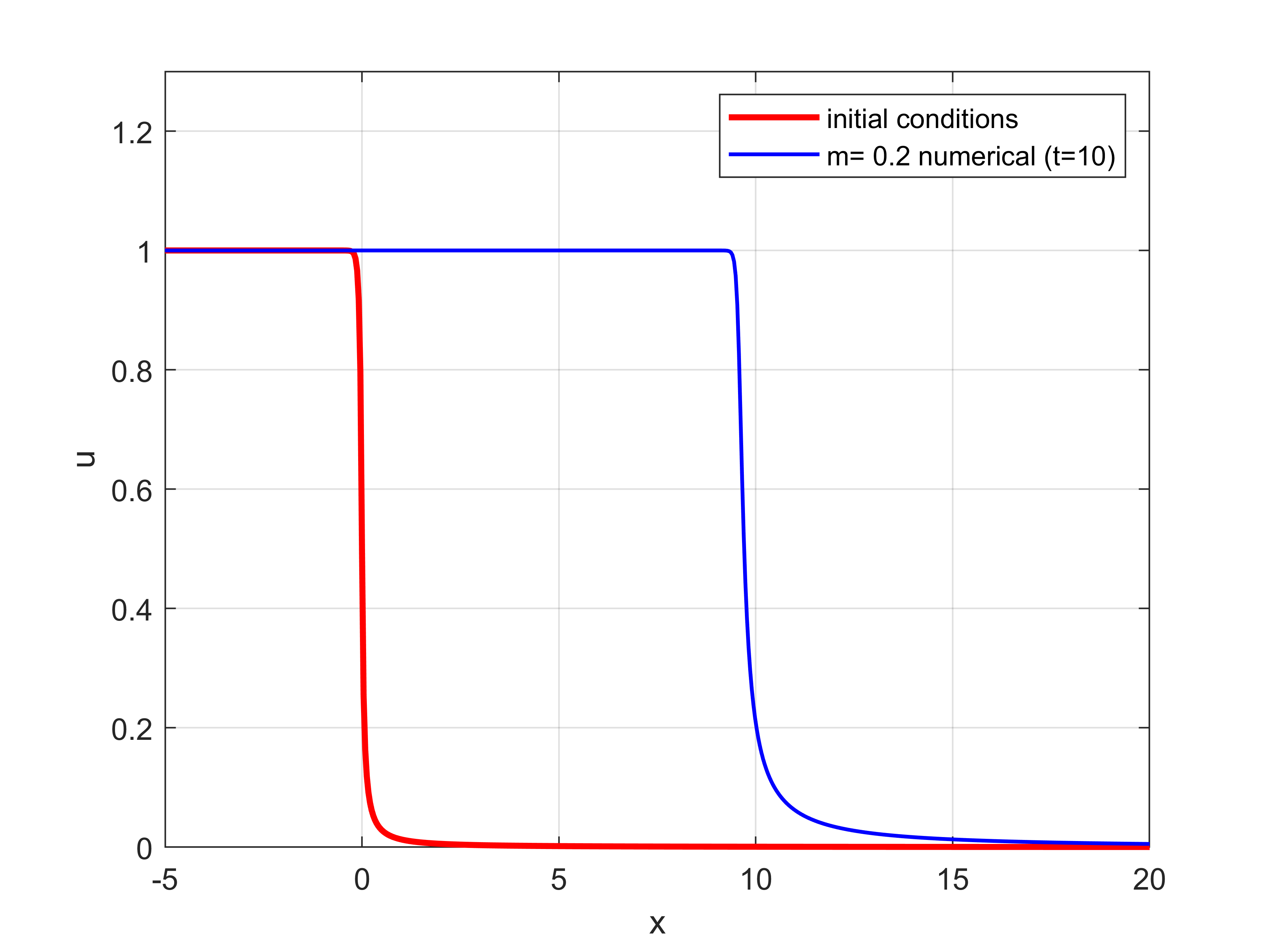}
		\caption{Case 2: $m=0.2$. The solution $u(x,t)$ behaves like a monotone viscous shock wave.}
		\label{fig2}
	\end{center}
\end{figure}

\noindent

{\bf Example 3.}~~
Let $f(u)=u^{2+2m}-u^{2m}$, for $\frac{1}{2}<m<1$. The targeted equation is
\begin{equation}\label{e3}
  u_t+(u^{2+2m}-u^{2m})_x=\mu (u^m)_{xx},\;\;x\in \rm{R},\,t>0.
\end{equation}
The solution possesses a unique (up to shift) viscous shock $U(x-st)$ with the speed $s=\frac{f(u_-)}{u_-}=0$ satisfying the degenerate entropy condition $f'(0)=s<f'(1)$. From \eqref{211}, the decay properties of the viscous shock at the far fields are:
\begin{eqnarray*}
&&|U(\xi)-0|=O(1)|\xi|^{-\frac{1}{2-m}} \ \ \mbox{ as } \xi\to +\infty, \\
&& |U(\xi)-1|=O(1)e^{-\frac{2}{m}|\xi|} \ \ \mbox{ as } \xi\to -\infty.
\end{eqnarray*}
Thus, we choose the initial data as
\begin{equation*}
   u_0(x)=\left\{
   \begin{array}{ll}
   \frac{1}{2}{\left(\frac{ 4-2m}{ m}x+1\right)^{-\frac{1}{2-m}}},~&x\geq 0,\\[6pt]
   1-\frac{1}{2}e^{\frac{2x}{ m}},~&x<0,
   \end{array}\right.
 \end{equation*}
 where $\frac{1}{2}<m<1$.  Here, $f(u)$ is non-convex. In fact, there exists a unique $m_2^*:=\sqrt{\frac{m(2m-1)}{(1+2m)(1+m)}}$ such that $f''(m_2^*)=0$ and
\begin{equation*}
  f''(u)\left\{
  \begin{array}{ll}
    <0,~u\in[0,m_2^*),\\[8pt]
     >0,~u\in(m_2^*,1].
  \end{array}\right.
\end{equation*}
We run our numerical computations by selecting $m=0.6$, \ 0.8, \ 0.9, respectively. Figures 3 confirms that the solutions $u(t,x)$ behaves like viscous shock waves.

\begin{figure}[htbp]
	\begin{center}
		\includegraphics[width=7cm]{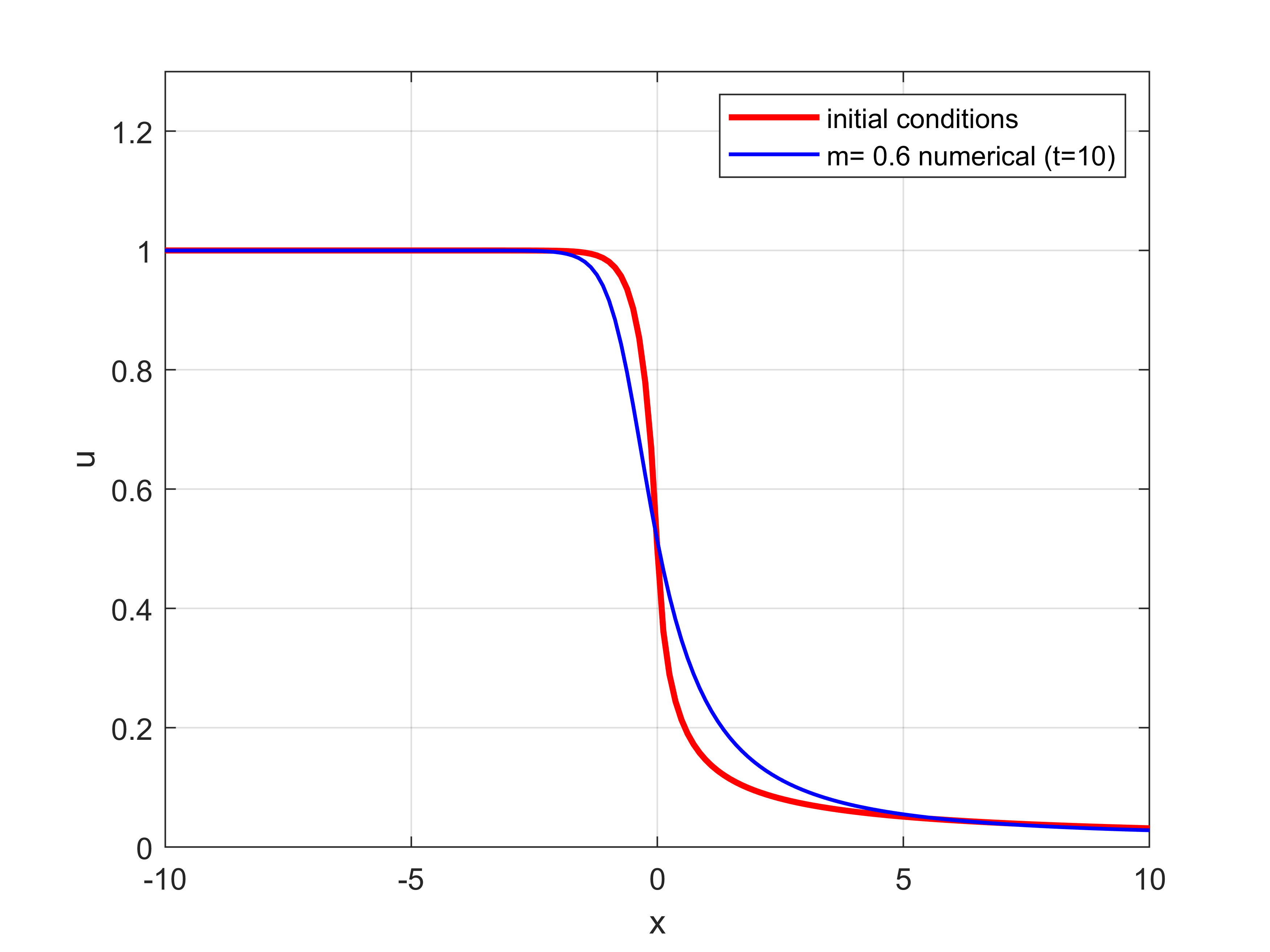}
		\includegraphics[width=7cm]{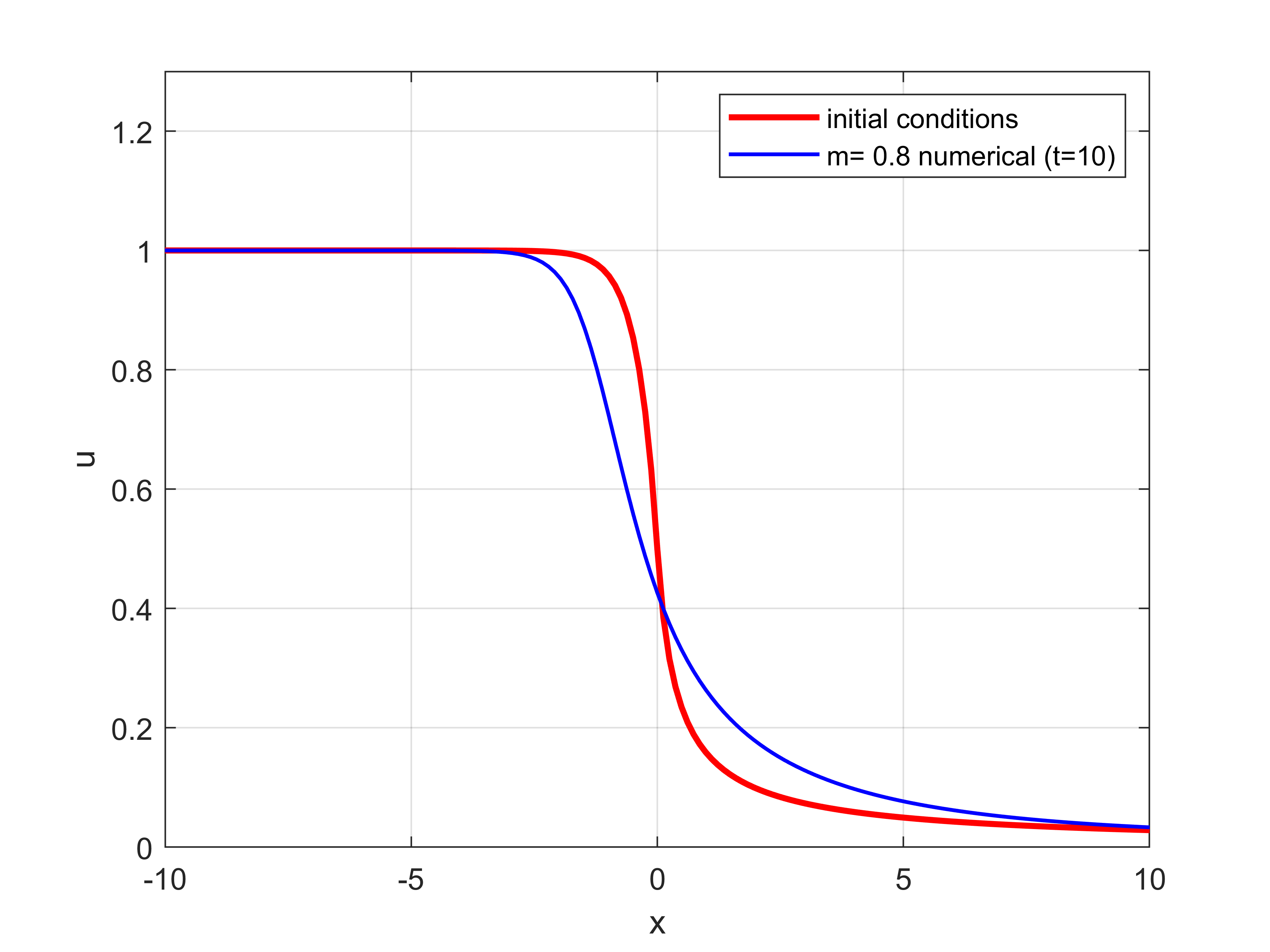}
        \includegraphics[width=7cm]{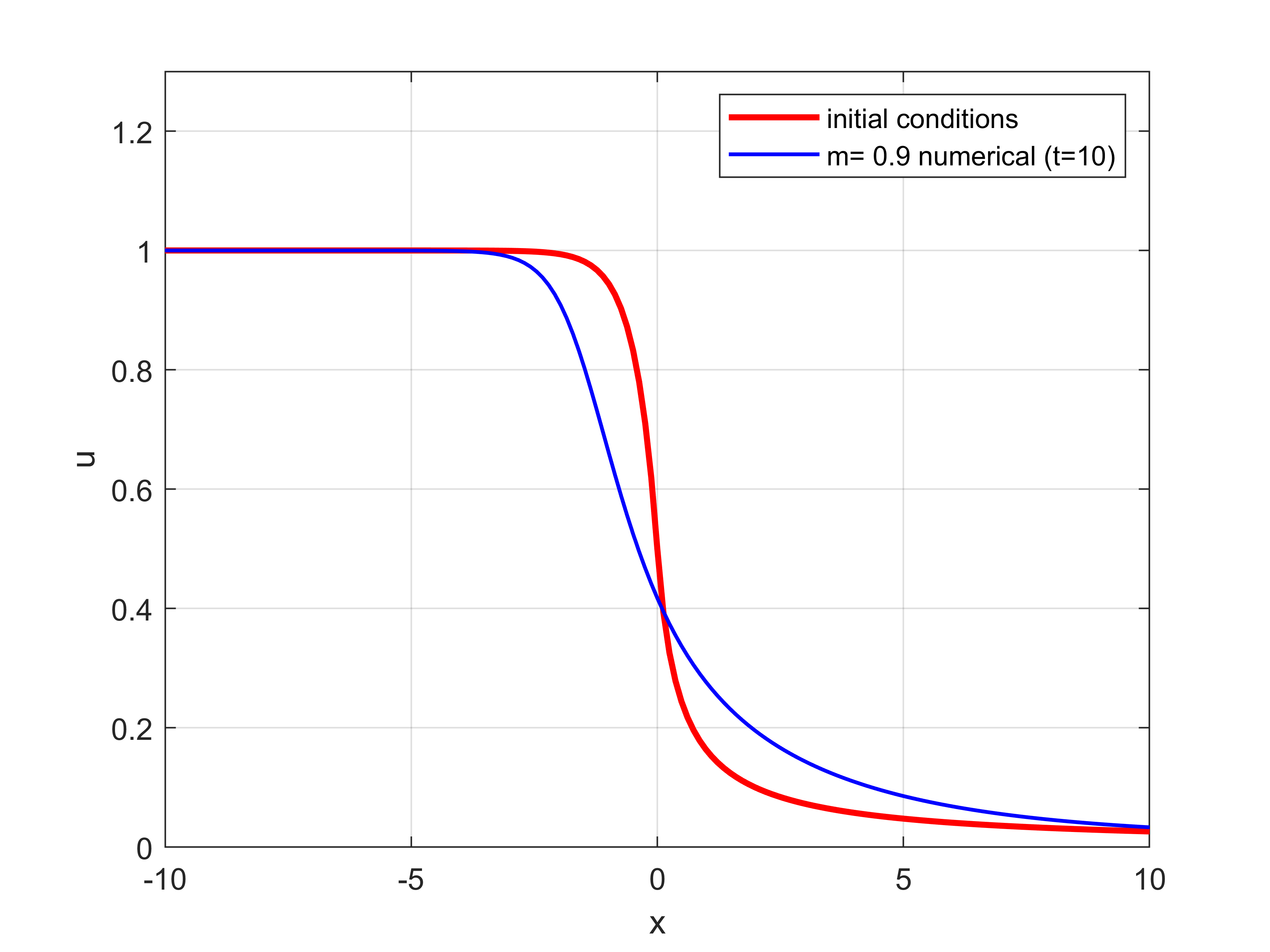}
		\caption{Case 3: $m=0.6, \ 0.8, \ 0.9$. The solution $u(x,t)$ behaves like a monotone viscous shock wave.}
		\label{fig3}
	\end{center}
\end{figure}

{\bf Example 4.}~~Let $f(u)=u^3-u^2$, for $m=\frac{1}{2}$.  The targeted equation is
\begin{equation}\label{e4}
  u_t+(u^3-u^2)_x=\mu (u^\frac{1}{2})_{xx},\;\;x\in \rm{R},\,t>0.
\end{equation}
The solution possesses a unique (up to shift) viscous shock wave $U(x-st)$ with the wave speed $s=\frac{f(u_-)}{u_-}=0$ satisfying  $f'(0)=s<f'(1)$. At the far fields $\xi\to\pm\infty$, this viscous shock wave behaves like
\begin{eqnarray*}
&&|U(\xi)-0|=O(1)|\xi|^{-\frac{3}{2}} \ \ \mbox{ as } \xi\to +\infty, \\
&& |U(\xi)-1|=O(1)e^{-2|\xi|} \ \ \mbox{ as } \xi\to -\infty.
\end{eqnarray*}
Thus, we choose the initial data as
\begin{equation*}
   u_0(x)=\left\{
   \begin{array}{ll}
   \frac{1}{2}{\left(3x+1\right)^{-\frac{2}{3}}},~&x\geq 0,\\[6pt]
   1-\frac{1}{2}e^{2x},~&x<0.
   \end{array}\right.
 \end{equation*}
The flux function $f(u)$ is non-convex:
\begin{equation*}
  f''(u)\left\{
  \begin{array}{ll}
    <0,~u\in[0,{\frac{{1}}{{3}}}),\\[8pt]
     >0,~u\in({\frac{{1}}{{3}}},1].
  \end{array}\right.
\end{equation*}
We carry out our numerical simulations for $m=0.5$. Figure 4 shows that the solution $u(t,x)$ behaves like a viscous shock wave.

\begin{figure}[htbp]
	\begin{center}
		\includegraphics[width=7cm]{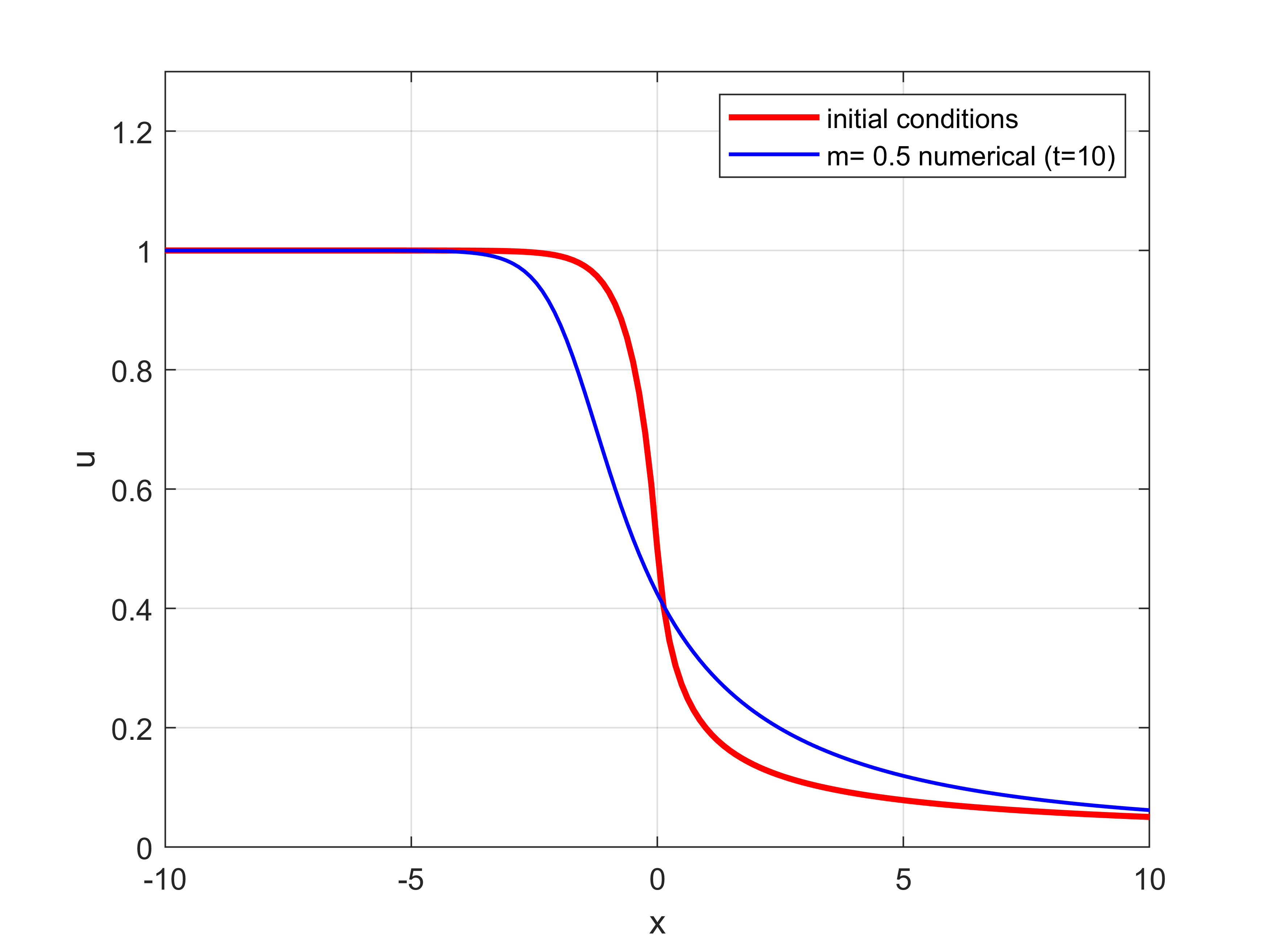}
		\caption{Case 4: $m=0.5$. The solution $u(x,t)$ behaves like a monotone viscous shock wave.}
		\label{fig1}
	\end{center}
\end{figure}

 \section{Acknowledgment}

The paper was done when S. Xu studied at McGill University as a PhD Trainee supported by China Scholarship Council (CSC)(No.202106890040). S. Xu expresses her sincere thanks for the hospitality of McGill University and CSC. The research of M. Mei was supported in part by NSERC Grant RGPIN-2022-03374.  The research of J.-C. Nave was supported in part by the NSERC Discovery Grant program. The research of W. Sheng was supported in part by NSFC of China No.12171305.

\bibliographystyle{amsplain}

\begin{thebibliography}{20}


\bibitem{bat} H. Bateman, Some recent researches on the motion of fluids, Monthly Weather Rev. 43 (1915), pp. 163-170.

\bibitem{Biro} Z. Bir'o, Stability of tarvelling waves for degenerate reaction-diffusion equations of
KPP-type, Advanced Nonlinear Studies, 2 (2002), 357-371.


\bibitem{bur1}J. M. Burgers, Mathematical examples illustrating relations occurring in the theory of turbulent fluid motion, Verh. Nederl. Akad. Wetensch. Afd. Natuurk. Sect. 1 17 (1939), no. 2, pp. 1-53.

\bibitem{bur2} J. M. Burgers, A mathematical model illustrating the theory of turbulence, Advances in Applied Mechanics, Academic Press, New York (1948).




	\bibitem{Engler} H. Engler,
	Asymptotic stability of traveling wave solutions for perturbations with algebraic decay, J. Differ. Equ., 185 (2002), 348-369.
	
	\bibitem{Serre} H. Freist\"uhler and D. Serre,
	$L^1$ stability of shock waves in scalar viscous conservation laws, Commun. Pure Appl. Math., 51 (1998),  291-301.
	
	\bibitem{Fries} C. Fries,
	Stability of viscous shock waves associated with non-convex modes, Arch. Ration. Mech. Anal., 152 (2000),  141-186.
	
	

\bibitem{GK} B.H. Gilding, R. Kersner, TravelingWaves in Nonlinear Diffusion-Convection Reaction, Springer Basel AG, 2004.

	
	\bibitem{Goodman} J. Goodman, Nonlinear asymptotic stability of viscous shock profiles for conservation laws, Arch. Ration. Mech. Anal., 95 (1986), 325-344.
	
	
	

	\bibitem{Howard-2} P. Howard,
	Pointwise Green's function approach to stability for scalar conservation laws, Commun. Pure Appl. Math., 52 (1999), 1295-1313.
	
	\bibitem{Howard-3} P. Howard,
	Pointwise estimates on the Green's function for a scalar linear convection-diffusion equation, J. Differ. Equ., 155 (1999),  327-367.
	
	\bibitem{Howard-1} P. Howard,
	Pointwise estimates and stability for degenerate viscous shock waves, J. Reine Angew. Math., 545 (2002), 19-65.

	
	
	
	\bibitem{Ili'in} A. Il'in and O. Oleinik, Asymptotic behavior of the solutions of the Cauchy problem for
	certain quasilinear equations for large time, Math. USSR. Sb., 51 (1960), 191-216.	
	
	
	\bibitem{Jones}	C. Jones, R. Gardner and T. Kapitula, Stability of traveling waves for non-convex scalar viscous conservation laws, Commun. Pure Appl. Math., 46 (1993), 505-526.

\bibitem{K-R-1} S. Kamin and P. Rosensau, Convergence to the travelling wave solution for a nonlinear reaction-diffusion equation, Rend. Mat. Acc. Lincei
 15 (2004), 271-280.

\bibitem{K-R-2} S. Kamin and P. Rosenau, Emergence of waves in a nonlinear convection-reaction-diffusion equation, Adv. Nonlin. Stud., 4 (2004), 251-272.

\bibitem{KV}  M.-J. Kang and A. Vasseur, $L^2$-contraction for shock waves of scalar viscous conservation laws, Annales de l'Institut Henri Poincar'e C, Analyse non lin'eaire
 34 (2017), 139-156.

	\bibitem{KVW} M.-J. Kang, A. Vasseur, Y. Wang, Time-asymptotic stability of composite waves of viscous shock and rarefaction for
	barotropic Navier-Stokes equations, Adv. Math., 419 (2023), 108963.
	
	\bibitem{growing interfaces} M. Kardar, G. Parisi and Y. Zhang, Dynamic scaling of growing interfaces, Phys. Rev. Lett., 56 (1986), 889-892.

	
	\bibitem{same line} S. Kawashima and A. Matsumura, Asymptotic stability of traveling wave solutions of systems for one-dimensional gas motion, Commun. Math. Phys., 101 (1985), 92-127.
	
	\bibitem{nonconvex} S. Kawashima and A. Matsumura, Stability of shock profiles in viscoelasticity with non-convex constitutive relations, Commun. Pure Appl. Math., 47 (1994), 1547-1569.
	
	\bibitem{Kim} Y. Kim and A. Tzavaras, Diffusive $N$-waves and metastability in the Burgers equation, SIAM J. Math. Anal., 33 (2001), 607-633.
	
	
	
	\bibitem{Li-Li-Mei-Nave} X. Li, J. Li, M. Mei, and J.-C. Nave, Nonlinear stability of shock profiles to Burgers equation with critical fast diffusion and singularity, preprint (2024).
	
	\bibitem{Liu} T. Liu, Nonlinear stability of shock waves for viscous conservation laws, Mem. Am. Math. Soc., 328 (1985), 1-108.
	
	\bibitem{IBVP2} T. Liu, A. Matsumura and K. Nishihara, Behavior of solutions for the Burgers equations with boundary corresponding to rarefaction waves, SIAM J. Math. Anal., 29
	(1998), 293-308.
	
	\bibitem{IBVP3} T. Liu and K. Nishihara, Asymptotic behavior for scalar viscous conservation laws with boundary effect, J. Differ. Equ., 133 (1997), 296-320.
	
	\bibitem{IBVP1} T. Liu and S. Yu, Propagation of a stationary shock layer in the presence of a boundary, Arch. Ration. Mech. Anal., 139 (1997), 57-82.

	\bibitem{Liu-Zeng-1} T. Liu and Y. Zeng,
	Time-asymptotic behavior of wave propagation around a viscous shock profile, Commun. Math. Phys., 290 (2009),  23-82.
	
	\bibitem{Liu-Zeng-2} T. Liu, Y.  Zeng,
	Shock waves in conservation laws with physical viscosity, Mem. Am. Math. Soc., 234 (2015), no.1105.

	\bibitem{MZ} C. Mascia and K. Zumbrun,
	Stability of large-amplitude viscous shock profiles of hyperbolic-parabolic systems, Arch. Ration. Mech. Anal., 172 (2004),  93-131.
	
	
	\bibitem{Matsumura-Mei} A. Matsumura and M. Mei, Convergence to travelling fronts of solutions of the p-system with viscosity in the presence of a boundary, Arch. Ration. Mech. Anal.,  146 (1999), 1-22.
	
	
	\bibitem{Matsumura-Nishihara-1} A. Matsumura, K. Nishihara, On the stability of traveling wave solutions of a
	one-dimensional model system for compressible viscous gas, Japan J. Appl. Math., 2
	(1985) 17-25.
	
	\bibitem{conventional energy method} A. Matsumura and K. Nishihara, Asymptotic stability of traveling waves
	for scalar viscous conservation laws
	with non-convex nonlinearity, Commun. Math. Phys., 165 (1994), 83-96.
	

	
	\bibitem{M. Mei} M. Mei, Stability of shock profiles for non-convex scalar viscous conservation laws, Math. Models Methods Appl. Sci., 05 (1995), 279-296.
	
	
	\bibitem{Nishihara} K. Nishihara, A note on the stability of traveling wave solutions of Burgers' equation, Japan J. Appl. Math., 2 (1985), 27-35.

	
	
	\bibitem{spectural} D. Sattinger, On the stability of waves of nonlinear parabolic systems, Adv. Math., 22 (1976), 312-355.

 \bibitem{cut1} Y. Sugiyama,
Partial regularity and blow-up asymptotics of weak solutions to degenerate parabolic systems of porous medium type, Manuscripta Math., 147 (2015), 311–363.

  \bibitem{cut2}Y. Sugiyama and Y. Yahagi,
Extinction, decay and blow-up for Keller-Segel systems of fast diffusion type, J. Differ. Equ., 250 (2011), 3047–3087.

	
	\bibitem{Szepessy} A. Szepessy and Z. Xin, Nonlinear stability of viscous shock waves, Arch. Ration. Mech. Anal.,
	122 (1993), 53-104.
	



	\bibitem{nonconvex3} H. Weinberger, Long-time behavior for a regularized scalar conservation law in the absence of genuine nonlinearity, Ann. Inst. H. Poincar\'{e} Anal. Non Lin\'{e}aire, 07 (1990), 407-425.
	

\bibitem{XJMY} T. Xu, S. Ji, M. Mei, and J. Yin, Stability of sharp traveling waves for Burgers-Fisher-KPP equations with degenerate diffusion, J. Nonlinear Sci., 34:44 (2024).
doi.org/10.1007/s00332-024-10021-x.







\end{thebibliography}

\end{document}